\documentclass[graybox]{svmult}

\usepackage{type1cm}         
\usepackage{makeidx}         
\usepackage{graphicx}        
\usepackage{multicol}        
\usepackage[bottom]{footmisc}

\usepackage{newtxtext}       %
\usepackage[varvw]{newtxmath}       

\makeindex             

\usepackage{float}
\usepackage{soul}
\usepackage{url}
\usepackage[many]{tcolorbox}
\usepackage{soul}
\usepackage{framed}
\usepackage{braket}
\usepackage{bbold}
\usepackage{enumerate} 
\usepackage{todonotes}
\usepackage{mathtools}
\DeclarePairedDelimiter\ceil{\lceil}{\rceil}
\DeclarePairedDelimiter\floor{\lfloor}{\rfloor}

\newcommand{\parencite}{\cite}
\newcommand\norm[1]{\left\lVert#1\right\rVert}
\usepackage{tikz}
\usepackage{tikz-cd}
\usepackage{hyperref}

\usepackage[utf8]{inputenc}
\usepackage{csquotes}

\begin{document}
\title*{Convergence of the micro-macro Parareal Method for a Linear Scale-Separated Ornstein-Uhlenbeck SDE: extended version}
\titlerunning{Convergence of micro-macro Parareal for a linear multiscale OU SDE}

\author{Ignace Bossuyt and 
Stefan Vandewalle and
Giovanni Samaey}
\institute{Ignace Bossuyt \at Department  of  Computer  Science,  KU  Leuven,  Celestijnenlaan 200A,  3001  Leuven,  Belgium, \, 
\email{ignace.bossuyt1@kuleuven.be}
\and Stefan Vandewalle \at Department  of  Computer  Science,  KU  Leuven,  Celestijnenlaan 200A,  3001  Leuven,  Belgium, \, \email{stefan.vandewalle@kuleuven.be}
\and Giovanni Samaey \at Department  of  Computer  Science,  KU  Leuven,  Celestijnenlaan 200A,  3001  Leuven,  Belgium, \, \email{giovanni.samaey@kuleuven.be}}
 
\maketitle 
 
\abstract*{Time-parallel methods can reduce the wall clock time required for the accurate numerical solution of differential equations by parallelizing across the time-dimension. In this paper, we present and test the convergence behavior of a multiscale, micro-macro version of a Parareal method for stochastic differential equations (SDEs). In our method, the fine propagator of the SDE is based on a high-dimensional slow-fast microscopic model; the coarse propagator is based on a model-reduced version of the latter, that captures the low-dimensional, effective dynamics at the slow time scales.
We investigate how the model error of the approximate model influences the convergence of the micro-macro Parareal algorithm and we support our analysis with numerical experiments.}

\textbf{Abstract}
Time-parallel methods can reduce the wall clock time required for the accurate numerical solution of differential equations by parallelizing across the time-dimension. In this paper, we present and test the convergence behavior of a multiscale, micro-macro version of a Parareal method for stochastic differential equations (SDEs). In our method, the fine propagator of the SDE is based on a high-dimensional slow-fast microscopic model; the coarse propagator is based on a model-reduced version of the latter, that captures the low-dimensional, effective dynamics at the slow time scales.
We investigate how the model error of the approximate model influences the convergence of the micro-macro Parareal algorithm and we support our analysis with numerical experiments.
This is an extended and corrected version of [Domain Decomposition Methods in Science and Engineering XXVII. DD 2022, vol 149 (2024), pp. 69-76, Bossuyt, I., Vandewalle, S., Samaey, G.].

\section{Introduction and motivation}
This paper is an extended and corrected version of \parencite{bossuyt_convergence_2024}. 
Our aim is to obtain insight in the convergence of a parallel-in-time (PinT) method applied to a two-dimensional linear stochastic differential equation (SDE).  
In our method, the fine propagator of the SDE is based on a high-dimensional slow-fast microscopic model; the coarse propagator is based on a model-reduced version of the latter, that captures the low-dimensional, effective dynamics at the slow time scales.
More specifically, we are interested in the behavior of the mean and of the variance of the SDEs, as a function of the micro-macro Parareal iterations. This allows for an analytic treatment. 
We expect this convergence analysis to be useful as a stepping stone for analyzing PinT methods for higher-dimensional (nonlinear) SDEs.

In section \ref{section_micro_macro_Parareal}, we discuss the micro-macro Parareal algorithm for linear multiscale ODES and we present some extensions to its existing convergence theory. 
In section \ref{section_model_problem} we recapitulate the model problem SDE from \cite{bossuyt_convergence_2024}. 
In section \ref{section_micro_macro_Parareal_for_OU} we present our corrected convergence bounds for the moments of the bivariate multiscale Ornstein-Uhlenbeck SDE. 
In section \ref{section_numerical_experiments} we numerically verify the assumptions made in the convergence analysis and we present the same numerical experiments from \cite{bossuyt_convergence_2024}.

\section{Convergence of the micro-macro Parareal Algorithm for a linear multiscale ODE}
\label{section_micro_macro_Parareal}
In this section, we briefly recap the micro-macro Parareal algorithm, and formulate some lemmas that will be useful in later sections.
The Parareal algorithm \parencite{lions_resolution_2001} is 
an iterative method for the numerical approximation of initial-value problems.
Micro-macro Parareal is a generalisation of Parareal, specifically for the simulation of scale-separated ODEs \parencite{Legoll2013} and for scale-separated SDEs \parencite{legoll_parareal_2020}.

We consider initial-value problems for the evolution of a state variable $u \in \mathbb R^l$, with $l$ a natural number, of the form $\frac{du}{dt} = f(u,t)$ where $f$ is specific to the problem and where the time variable $t$ is contained in the interval $t \in [0,T]$.
We consider a discretisation in time on $\Delta t$-equispaced points $t_n = n \Delta t$, with $0 \leq n \leq N$ and such that $t_N = T$. 
We denote the approximations of the solution $u$ at time $t_n$ with $u_n$. 

The micro-macro Parareal algorithm combines two levels of description: (i) the micro variable $u$, which is high-dimensional, and (ii) the macro variable $X$, which is lower-dimensional.
We will denote the effect of a numerical approximation for the evolution of the micro variable with the fine propagator $\mathcal{F}_{\Delta t}$, advancing the solution between two consecutive time points. 
Similarly, we define the coarse propagator $\mathcal{C}_{\Delta t}$.

The micro and macro variables are related through coupling operators: 
the restriction operator $\mathcal{R}$ extracts macro information from a micro state, 
the lifting operator $\mathcal{L}$ produces a micro state that is consistent with a given macro state, 
and, thirdly, the matching operator $\mathcal{M}$ produces a micro state that is consistent with a given macro state, based on prior information of the micro state.

The micro-macro Parareal algorithm iterate at iteration $k$ and time step $n$ is given next.
For $k=0$ (initialization), we have, for $0 \leq n \leq N-1$,
\begin{equation}
\begin{aligned} 
X_0^0 = X_0 
\qquad 
X_{n+1}^{0} 
&= \mathcal{C}_{\Delta t} (X_n^{0}) 
\qquad 
u^0_0 = u_0
\qquad
u_{n+1}^{0} 
= \mathcal{L}(X_{n+1}^{0}),
\end{aligned}
\label{mM_Parareal_0}
\end{equation}
and for $k\geq 1$,
\begin{equation}
\begin{aligned} 
X^{k+1}_0 = X_0 \qquad 
X_{n+1}^{k+1} 
&= \mathcal{C}_{\Delta t} (X_n^{k+1}) 
+ \mathcal{R} ( \mathcal{F}_{\Delta t} (u_n^k) )
- \mathcal{C}_{\Delta t}( X_n^k) \\
u^{k+1}_0 = u_0 \qquad      
u_{n+1}^{k+1} &= \mathcal{M}(X_{n+1}^{k+1}, \mathcal{F}_{\Delta t} (u_n^k)).
\end{aligned}
\label{mM_Parareal_other_k}
\end{equation}
If the coupling operators are chosen such that $\mathcal{M}(\mathcal{R}u,u) = u$, then at each iteration it holds that $X_n^k = \mathcal{R}u_n^k$.
Classical Parareal \parencite{lions_resolution_2001} corresponds to the case $\mathcal{R} = \mathcal{L} = \mathcal{M} = \mathcal{I}$.

\subsection{Linear multiscale ODE system and its properties}
In \parencite{Legoll2013}, the convergence of micro-macro Parareal for a linear scale-separated ODE is studied. We briefly review the main ingredients of that theory, because we will use them further on to study the convergence for our model problem \eqref{slow_fast_OU_full}.

The test system in \parencite{Legoll2013} models the coupled evolution of a slow variable $x \in \mathbb{R}$ and a fast variable $y \in \mathbb{R}^d$, \textcolor{black}{$d\geq 1$}.
The evolution of $u= (x,y)$ for $t \in [0, T]$ is given by this linear multiscale ODE:
\begin{equation}
\left[ \begin{matrix} 
\dot x \\ \dot y 
\end{matrix} \right]
= \left[ 
\begin{matrix} 
a & p^T \\ q/\epsilon & -A/\epsilon \end{matrix} \right]
\left[ \begin{matrix} x \\ y \end{matrix} \right],
\qquad
\textcolor{black}{
\left[ \begin{matrix} x(0) \\ y(0) \end{matrix} \right]
=
\left[ \begin{matrix} x_0 \\ y_0 \end{matrix} \right]}
\label{test_system_Samaey_2013}
\end{equation}
where \textcolor{black}{$A \in \mathbb R^{p \times p}$} has positive eigenvalues \textcolor{black}{$\mu \geq \mu_- > 0$};
in words, the fast component $v$ is dissipative.
The reduced model for the approximate slow variable $X$ is:
\begin{equation}
\dot{X} =  \Lambda X,
\qquad
\textcolor{black}{X(0) = x_0}
\label{definition_lambda}
\end{equation} 
with $t \in [0,T]$ and $\Lambda = a + p^TA^{-1}q$.
We first copy a result from \parencite{Legoll2013}.

\begin{lemma}[Properties of the multiscale system \eqref{test_system_Samaey_2013} and its reduced model \eqref{definition_lambda}] 
For the initial-value problem \eqref{test_system_Samaey_2013} and its reduced model \eqref{definition_lambda}, there exists $\epsilon_0 \in (0,1)$, and a constant $C > 0$, independent of $\epsilon$, such that, for all $\epsilon < \epsilon_0$,

\begin{equation}
\sup_{t \in [0,T]} |x(t) - X(t)| \leq C \epsilon (|x_0| +  \| y_0 - A^{-1}q x_0 \| ),
\label{my_2013_eq_2_8}
\end{equation}
\begin{equation}
\sup_{t \in [0,T]} |x(t)| \leq C (|x_0| + \epsilon \| y_0 \|),
\label{my_2013_eq_2_13}
\end{equation}
\begin{equation}
\sup_{t \in [t^{\mathrm{BL}}_{\epsilon},T]} 
\|y(t)\| \leq C (|x_0| + \epsilon \| y_0 \|),
\label{my_2013_eq_2_14}
\end{equation}
where $t^{\mathrm{BL}}_{\epsilon}$ is the length of a boundary layer in time of the order of $\epsilon$:
\begin{equation}
t^{\mathrm{BL}}_{\epsilon} = \frac{2 \epsilon}{\mu_-} \ln\left( \frac{1}{\epsilon} \right).
\end{equation}
\label{lemma_properties_multiscale_system_2013}
\end{lemma}
\begin{proof}
The bounds are proved in \parencite[proof of lemma 2 and corollary 3]{Legoll2013}.
\end{proof}

\subsection{Convergence of micro-macro Parareal for a linear multiscale test ODE}
This section is heavily based on \parencite{Legoll2013}.
We consider the following coupling operators, relating the micro variable $u = (x, \, y)$, obeying the multiscale evolution law in equation \eqref{test_system_Samaey_2013} and the macro variable $X$ obeying equation \eqref{definition_lambda}:
\begin{equation}
\begin{aligned}
\mathcal R \left( \left( \begin{matrix} x,y \end{matrix} \right) \right) &= x \\
\mathcal{M}(X,u) &= 
\left( \begin{matrix}
X, & \mathcal{R}^{\perp} u \end{matrix}  \right) \\
\mathcal{L}(X) &= 
\left( \begin{matrix} X, & A^{-1}q  X  \end{matrix} \right) \\
\end{aligned}
\label{definition_classical_coupling_operators}
\end{equation}
where $\mathcal R^{\perp} \left( u \right)
= 
\mathcal R^{\perp} \left( \left( \begin{matrix} x,y \end{matrix} \right) \right) 
= y$. 
In words, the restriction operator $\mathcal R$ discards the fast variable. 
The matching operator $\mathcal M$ modifies the micro variable such that its slow component equals a desired macro variable.
The lifting operator $\mathcal L$ puts the fast component of the micro variable to its equilibrium, conditional on the slow variable.

\begin{property}[Continuity property of the matching operator]
The matching operator satisfies
\begin{equation}
\norm{ \mathcal M \left( X, \left[ \begin{matrix}
x \\ y
\end{matrix} \right] \right) }
\leq 
\|X\| + \|y\|
\end{equation}
\label{continuity_property_matching}
\end{property}
\begin{proof}
The result follows from 
the definition of the matching operator $\mathcal M$ in equation \eqref{definition_classical_coupling_operators}, and the triangle inequality.
\end{proof}

Using the properties \textcolor{black}{in equations} \eqref{my_2013_eq_2_8} - \eqref{my_2013_eq_2_14} in \parencite{Legoll2013}, we will analyse the convergence of micro-macro Parareal for the linear test problem \eqref{test_system_Samaey_2013} with coarse model \eqref{definition_lambda}.
We first copy a theorem from \parencite{Legoll2013} to give some context.

\begin{lemma}[Convergence of micro-macro Parareal for homogeneous linear \textcolor{black}{multiscale} ODEs \eqref{test_system_Samaey_2013}-\eqref{definition_lambda}]
\label{lemma_micro_macro_convergence}
We use micro-macro Parareal, defined in \eqref{mM_Parareal_0}-\eqref {mM_Parareal_other_k}, 
with the coarse and fine propagators $\mathcal C_{\Delta t}$ and $\mathcal F_{\Delta t}$ as the exact solutions of equations \eqref{test_system_Samaey_2013}-\eqref{definition_lambda}, 
and with the coupling operators restriction from equation \eqref{definition_classical_coupling_operators}.
Let the exact solution $u_n$ satisfy $u_{n+1} = \mathcal F u_n$, and let $E_{n}^k = U_n^k - \mathcal{R} u_n$ be the macro error and $e_n^k = u_n^k - u_n$ be the micro error. 
\textcolor{black}{
Then, there exists $\epsilon_0 \in (0,1)$, that only depends on $\alpha$, $p$, $q$, $A$ and $T$, such that, for all $\epsilon < \epsilon_0$ and all $\Delta t > t_\epsilon^{BL}$, there exists a constant $C_k > 0$, independent of $\epsilon$, such that for all $k \geq 0$:
}
\begin{equation}
\sup_{0 \leq n \leq N} |E_n^k| \leq C_k \epsilon^{1+\ceil{k/2}}
\end{equation}
\begin{equation}
\sup_{0 \leq n \leq N} \|e_n^k \| \leq C_k \epsilon^{1+\floor{k/2}}
\end{equation}
\end{lemma}
\begin{proof}
The proof is the same as \parencite[Proof of theorem 13]{Legoll2013}. 
\end{proof}

\subsection{Convergence of micro-macro Parareal for a linear multiscale test ODE with another lifting operator}
We now present a slightly modified version of lemma \ref{lemma_micro_macro_convergence}. The only difference is a different choice of lifting operator, namely $\mathcal{L}(X) = \left[ \begin{matrix} X, y_0 \end{matrix} \right]$ instead of $\mathcal{L}(X) = \left[ \begin{matrix} X, & A^{-1}q  X  \end{matrix} \right]^T$.
The latter choice initialises the fast variable to its equilibrium, conditional on the slow variable.
The alternative lifting operator, which initialises the fast variable to its value at $t=0$, is useful, for instance, if the conditional equilibrium value is not known, or too expensive to compute. 
In our numerical experiments, we choose this alternative lifting operator, since for general (nonlinear) SDEs it may be impossible to cheaply obtain information about their steady-state distribution..

\begin{lemma}[Convergence of micro-macro Parareal with \textcolor{black}{lifting based on the initial condition} in the zeroth iteration for homogeneous linear \textcolor{black}{multiscale} ODEs]
We use micro-macro Parareal, defined in \eqref{mM_Parareal_0}-\eqref {mM_Parareal_other_k}, 
with the coarse and fine propagators
$\mathcal C_{\Delta t}$ and $\mathcal F_{\Delta t}$ the exact solution of the ODEs
\eqref{test_system_Samaey_2013}-\eqref{definition_lambda},  
and with the coupling operators $\mathcal R$ and $\mathcal M$ from equation \eqref{definition_classical_coupling_operators}. 
For lifting, we use $\mathcal{L}(X) = \left[ \begin{matrix} X, y_0 \end{matrix} \right]$.
Let \textcolor{black}{$E_{n}^k = U_n^k - \mathcal{R} u_n$} be the macro error and $e_n^k = u_n^k - u_n$ be the micro error. 
Then, there exists $\epsilon_0 \in (0,1)$, that only depends on $\textcolor{black}{a}$, $p$, $q$, $A$ and $T$, and a constant $C_k$, independent of $\epsilon$, such that, for all $\epsilon < \epsilon_0$, all $\Delta t > t_\epsilon^{\mathrm{BL}}$ and all $k \geq 0$:
\textcolor{black}{
\begin{equation}
\sup_{0 \leq n \leq N} |E_n^k| \leq C_k \epsilon^{\ceil{(k+1)/2}},
\label{equation_convergence_no_lifting_macro}
\end{equation}
\begin{equation}
\sup_{0 \leq n \leq N} \|e_n^k \| \leq C_k \epsilon^{\floor{(k+1)/2}}.
\label{equation_convergence_no_lifting_micro}
\end{equation}
}
\label{lemma_no_lifting}
\end{lemma}

\noindent
The main difference of this lemma with respect to lemma \ref{lemma_micro_macro_convergence} is that the order of convergence in $\epsilon$ is lower. 
The proof proceeds as follows: once the bound is proven in the zeroth iteration for this special lifting operator, the inductive proof technique from \parencite[proof of theorem 13]{Legoll2013} is applied.

\begin{proof}
The flow of the proof of this lemma closely follows \parencite[proof of theorem 13]{Legoll2013}. 
More specifically, in the zeroth iteration, using equation \eqref{my_2013_eq_2_8},
there exists a constant $C$, independent of $\epsilon$, such that the macro error can be bounded as follows: 
\begin{equation}
\begin{aligned}
\sup_{0 \leq n \leq N} |E^0_n|
&= \sup_{0 \leq n \leq N} |r(t_n) - U(t_n)| 
&\leq C \epsilon
\end{aligned}
\end{equation}
For the micro error in the zeroth iteration, there exists a constant $C$ such that, using property \ref{continuity_property_matching}
\begin{equation}
\begin{aligned}
\sup_{0 \leq n \leq N}\|e^0_n\| 
&=  \sup_{0 \leq n \leq N} \left[ \begin{matrix}
r(t_n) \\ v(t_n)
\end{matrix} \right] - 
\left[ \begin{matrix}
U(t_n)  \\ v_0
\end{matrix} \right] \\
&= \sup_{0 \leq n \leq N}  
\left[ \begin{matrix}
E^0_n  \\ v(t_n) - v_0
\end{matrix} \right] \\
&\leq 
\sup_{0 \leq n \leq N} \left( |E^0_n| + \| v(t_n) - v_0 \| \right) \\
&\leq 
\sup_{0 \leq n \leq N} \left( |E^0_n| + C \left( |r_0| +  \epsilon\| v_0 \| \right) + \| v_0 \| \right) 
\end{aligned} 
\end{equation}
where we used the triangle inequality, as well as equation \eqref{my_2013_eq_2_14}, to go from the third to the fourth line.
Thus there exist constants $C$ and $K$ such that
\begin{equation}
\sup_{0 \leq n \leq N}\|e^0_n\| 
\leq
K + C \epsilon
\end{equation}
Since $\epsilon < 1$, one can then find another constant $C$ such that$ 
\sup_{0 \leq n \leq N}\|e^0_n\|  
\leq 
C$.
The quantity $\sup_{0 \leq n \leq N} \|e^0_n\|$ is thus independent of $\epsilon$. 
Intuitively this makes sense. The parameter $\epsilon$ determines the rate of convergence to the steady-state but it does not influence how far the initial condition lies from the steady state.
Thus, at $k=0$, it holds that $\ceil{(k+1)/2} = 1$ and $\floor{(k+1)/2} = 0$, and therefore equation \eqref{equation_convergence_no_lifting_macro}
and \eqref{equation_convergence_no_lifting_micro} hold for $k=0$.

In the subsequent iterations $k \leq 1$, the inductive proof technique from \parencite[proof of theorem 13]{Legoll2013} can be applied.
This proof makes use of
the relations \parencite[(4.23) and (4.25)]{Legoll2013}, which are recalled here for convenience: there exists a constant $C$ such that
\begin{equation}
\| E^{k+1}_n \| \leq C \epsilon \sum_{p=1}^{n-1} A_{\mathcal{C}_{\Delta t}}^{n-p-1} \left( |E^k_{p}| + \| e^k_p \| \right) 
\label{SISC_2013_4_25}
\end{equation}
\begin{equation}
\| e^{k+1}_n \| \leq C \left( |E^k_{n-1}| + \epsilon \|e^k_{n-1} \| + |E^{k+1}_n| \right) 
\label{SISC_2013_4_23}
\end{equation}
The derivation of these equations is based on a continuity property of the matching operator as well as a recursion for the macro state.
The rest of the proof by induction proceeds as follows: using the induction hypothesis for iterations $k-1$ and $k$, the bound is proved for iteration $k+1$, using equations \eqref{SISC_2013_4_23} and \eqref{SISC_2013_4_25}.
In the table below we summarize this procedure. 
\begin{table}[H]
\centering
\begin{tabular}{c||c|c||c|c|}
& \multicolumn{2}{c||}{$k$ even (put $m=k/2$)} & \multicolumn{2}{c}{$k$ odd (put $m=(k+1)/2$)} \\ \hline
upper bound for & $\sup_{0 \leq n \leq N} E^k_n$ & $\sup_{0 \leq n \leq N} e^k_n$ & $\sup_{0 \leq n \leq N} E^k_n$ & $\sup_{0 \leq n \leq N} e^k_n$ \\ \hline
iteration $k$ (by ind. hypothesis) 
	& $C_{k}\epsilon^{m+1}$ & $C_{k}\epsilon^{m}$ 
	& $C_{k}\epsilon^m$ & $C_{k}\epsilon^{m}$ \\ \hline
iteration $k+1$ 
	& $C_{k+1}\epsilon^{m+1}$ & $C_{k+1}\epsilon^{m+1}$ 
	& $C_{k+1}\epsilon^{m+1}$ & $C_{k+1}\epsilon^{m}$ \\ 
\end{tabular}
\end{table}
This ends the proof.
\end{proof}

\section{Generalization of convergence theory of Parareal for linear ODEs to affine ODEs}
Lemma \ref{lemma_micro_macro_convergence} and \ref{lemma_no_lifting} deal with micro-macro Parareal for linear multiscale ODEs. 
Later we will be interested in the convergence of micro-macro Parareal for affine multiscale ODEs (possibly containing a inhomogeniety).
In the next lemma we prove that, except for the zeroth iteration, the inhomogeneity does not influence the convergence of the error of micro-macro Parareal (and thus also for classical Parareal).

\begin{lemma}[Error propagation property of micro-macro Parareal for nonhomogenous linear ODEs with constant coefficients]
We use micro-macro Parareal, defined in \eqref{mM_Parareal_0}-\eqref {mM_Parareal_other_k}.
We use (numerical) fine and coarse propagators $\mathcal F$ and $\mathcal C$, advancing the solution over a time interval $\Delta t$, that satisfy the following properties: 
$\mathcal{F}_{\Delta t}(u,t) = A_{\mathcal{F}, {\Delta t}}u + B_{\mathcal{F},{\Delta t}}(t)$ and 
$\mathcal{C}_{\Delta t}(X,t) = A_{\mathcal{C},{\Delta t}} X + B_{\mathcal{C},{\Delta t}}(t)$, where $A_{\mathcal{F},{\Delta t}}$ and $A_{\mathcal{C},{\Delta t}}$ are constant matrices that depend on $\Delta t$.
$B_{\mathcal{F},{\Delta t}}(t)$ and $B_{\mathcal{C},{\Delta t}}(t)$ are functions that depend on $t$ and $\Delta t$.
The dependence on $\Delta t$ is later omitted for brevity in the notation. 
Then, for all $k\geq 1$, the following recursion relation is valid:
\begin{equation}
\begin{aligned}
U^{k+1}_n &= \mathcal R A_{\mathcal F} u^k_{n-1} + \mathcal R B_{\mathcal F}(t_n)
+ \sum_{p=1}^{n-1} A_{\mathcal C}^{n-p}
\left(
\mathcal R A_{\mathcal F} u^k_{p-1} + \mathcal R B_{\mathcal F}(t_p) - A_{\mathcal C} U^k_p
\right)
\end{aligned}
\label{correction_solution_Parareal}
\end{equation}
For the macro error on $U^{k+1}_n$, i.e., $E^{k+1}_n = U^{k+1}_n  - U_n$, we have
\begin{equation}
\begin{aligned}
E^{k+1}_n 
&= \mathcal R A_{\mathcal F} e^k_{n-1} 
+ \sum_{p=1}^{n-1} 
A_{\mathcal C}^{n-p}
\left(
\mathcal R A_{\mathcal F} e^k_{p-1} - A_{\mathcal C} E^k_p
\right) \\
&= \sum_{p=1}^{n-1} A_{\mathcal C}^{n-p-1} 
\left( \mathcal R A_{\mathcal F} e^{k}_{p} 
- \sum_{p=1}^{n-1} A_{\mathcal C}^{n-p} E^k_p \right)
\end{aligned}
\label{correction_error_Parareal}
\end{equation}
\label{convergence_Parareal_nonhomog_ODE}
\end{lemma}
The lemma is proven in appendix \ref{proof_of_affine_Parareal}.
\newline
In words, the recursion relation for the error of micro-macro Parareal for a linear ODE does not depend on the inhomogeneity for iterations $k \geq 1$.
Clearly, this result is also valid for classical Parareal (corresponding to $\mathcal R = \mathcal M = \mathcal L = I$).
In the zeroth, the macro approximation is given by 
\begin{equation}
\begin{aligned}
U^0_{n+1} 
= \mathcal C (U^0_n) 
&= A_{\mathcal{C}} U^0_n + B_{\mathcal{C}}(t_n) \\
&= A_{\mathcal{C}}^n U_0 +
\sum_{p=1}^n A_{\mathcal{C}}^{p} 
B_{\mathcal{C}}(t_{n-p})
\end{aligned}
\end{equation}
The exact micro solution is given by 
\begin{equation}
\begin{aligned}
u_{n+1} 
= \mathcal F (u^0_n) 
&= A_{\mathcal{F}} u_n + B_{\mathcal{F}}(t_n) \\
&= A_{\mathcal{F}}^n u_0 +
\sum_{p=1}^n A_{\mathcal{F}}^{p} 
B_{\mathcal{F}}(t_{n-p})
\end{aligned}
\end{equation}
The macro error in the zeroth iteration is thus given by
\begin{equation}
\begin{aligned}
E^0_n 
&=
U^0_n - Ru_n
=
U^0_n - U_n \\
&=  
A_{\mathcal{C}}^n 
U_0
- 
\mathcal R \left(
A_{\mathcal{F}}^n 
u_0 
\right)
+ 
\sum_{p=1}^n 
\left(
A_{\mathcal{C}}^{p} 
B_{\mathcal{C}}(t_{n-p})
- 
\mathcal R \left(
A_{\mathcal{F}}^{p} 
B_{\mathcal{F}}(t_{n-p}) 
\right)
\right)
\end{aligned}
\end{equation}
and it can be seen that, in the zeroth iteration, $E^0_n$ effectively depends on the inhomogeniety of the ODE.
The precise form for the micro error depends on the lifting operator.

\section{Stochastic differential equation model problem}
\label{section_model_problem}
In the remainder of this work, we consider a two-dimensional slow-fast Ornstein-Uhlenbeck (OU) stochastic differential equation (SDE) \parencite{uhlenbeck_theory_1930}.  
In this section we discuss this model problem and we derive a reduced model (section \ref{subsection_reduced_model}). 
We give ODEs that describe the first two moments of the SDE (subsection \ref{subsection_moments_OU}).
We then provide a detailed overview of the operators that will be used in micro-macro Parareal (subsection \ref{subsection_Parareal_for_OU}), and we list the required assumptions on parameters of the SDE such that our theoretical analysis can be applied (subsection \ref{section_overview_assumptions}).

We study a multiscale OU SDE that models the coupled evolution of a slowly evolving variable $x \in \mathbb{R}$ and a ‘fast' variable $y \in \mathbb{R}$ that quickly reaches its equilibrium distribution:

\begin{equation}
\begin{aligned}
\left[ \begin{matrix}
dx \\ dy
\end{matrix} \right] 
&= 
\left[ \begin{matrix}
\alpha & \beta \\
\gamma/\epsilon & \zeta/\epsilon
\end{matrix} \right] 
\left[ \begin{matrix}
x \\ y
\end{matrix} \right] dt
+ 
\sigma
\left[ \begin{matrix}
1 & 0 \\
0 & 1 /\sqrt{\epsilon}
\end{matrix} \right] 
dW.
\end{aligned}
\label{slow_fast_OU_full}
\end{equation} 
Here, $dW \in \mathbb{R}^2$ is a two-dimensional Brownian motion and $\epsilon \in \mathbb{R}$ is a (small) time scale separation parameter $\epsilon \ll 1$, and time $t \in [0,T]$. The initial condition has a distribution with mean $\left[ \begin{matrix}
m_{x,0} \\ m_{y,0} 
\end{matrix} \right]$ and covariance matrix $\left[ \begin{matrix}
\Sigma_{x,0} &  \Sigma_{xy,0}  \\  \Sigma_{xy,0} & \Sigma_{y,0} 
\end{matrix} \right]$.

Model problem \eqref{slow_fast_OU_full} mimics the general situation where $x$ is a low-dimensional quantity of interest whose evolution is influenced by a quickly evolving, high-dimensional variable $y$, described by SDEs. 
The joint probability density of $x$ and $y$ obeys an advection-diffusion partial differential equation, the Fokker-Planck equation (see, e.g. \parencite{gardiner_handbook_1983}). 
The direct solution of this partial differential equation using classical deterministic techniques, suffers from the curse of dimensionality. Instead,  one can obtain an approximation of the Fokker-Planck equation by using a Monte Carlo method on the corresponding SDE (see, e.g. \parencite{gardiner_handbook_1983}).
In this work, we are interested in the first moments (mean and (co)variance) of the SDE \eqref{slow_fast_OU_full}.

\subsection{Derivation of a reduced model for multiscale SDEs}
\label{subsection_reduced_model}
The averaging technique from \parencite[chapter 10, see, e.g., Remark 10.2]{pavliotis_multiscale_2008} allows to define a reduced dynamics for a scalar variable, that approximates the slow variable $x$ in \eqref{slow_fast_OU_full}. 
This averaging technique exploits time-scale separation to integrate out the fast variable with respect to $U^{\infty}(y|x)$, the invariant distribution of the fast variable $y$ conditioned on a fixed slow variable $x$.

The reduced model for the reduced variable $X \in \mathbb{R}$ reads as follows ($\lambda_{\Sigma}$ and $\Sigma_{\Sigma}$ are defined implicitly):
\begin{equation} 
d X = A(X) dt + S(X) dW \\[-\baselineskip]
\label{reduced_stochastic}
\end{equation} 
with 
\begin{equation*}
\begin{aligned}
A(X) &= \int_{\mathcal{Y}} a(X,y) U^{\infty}(y|X) dy
= \lambda_{\Sigma} X \coloneqq \left( \alpha - \frac{\beta \gamma}{\zeta} \right) X  \\
S(X) S(X)^T &= \int_{\mathcal{Y}} s(X,y) s(X,y)^T U^{\infty}(y|X) dy 
= \Sigma_{\Sigma}\textcolor{black}{^2} 
\coloneqq \sigma\textcolor{black}{^2},
\end{aligned}
\end{equation*}
where $\mathcal{Y}$ denotes the domain of $y$. It can be shown that for the OU system \eqref{slow_fast_OU_full}, the conditional distribution $U^{\infty}(y|x) = \mathcal{N}\left( \frac{\gamma x}{\zeta}, \frac{\sigma^2}{2 \zeta} \right)$ (see \parencite[Example 6.19]{pavliotis_multiscale_2008}). 

\textcolor{black}{
In summary, the reduced SDE approximating the dynamics of the slow variable $x$ of equation \eqref{slow_fast_OU_full} is
\begin{equation}
dX = \left(\alpha - \frac{\beta \gamma}{\zeta}  \right) X dt + \sigma dW.
\label{extra_clarity_SDE_slow}
\end{equation}
}

Although the reduced model \eqref{reduced_stochastic} is only an approximation to the slow dynamics, it offers two computational advantages w.r.t. the full, scale-separated system \eqref{slow_fast_OU_full}: (i) it contains fewer degrees of freedom, and (ii) it can be discretised with a time step that is independent of $\epsilon$. 
As $\epsilon$ approaches zero, the multiscale model \eqref{slow_fast_OU_full} becomes more stiff, while the (cheaper) reduced model becomes a more accurate approximation.

\subsection{Moments of the Ornstein-Uhlenbeck process and its reduced model} 
\label{subsection_moments_OU}
The evolution of the mean and the variance of a linear SDE
can be described exactly using the moment equation from \parencite{arnold_stochastic_1974}. 
Thus, for studying the moments of the linear Ornstein-Uhlenbeck SDE model problem, we can use linear ODEs instead of a Monte Carlo simulation.

\noindent
\textbf{\textit{Statistical moments of multiscale model.}}
The evolution of the mean of the multiscale SDE \eqref{slow_fast_OU_full} is described by the following linear ODE:
\begin{equation}
\frac{d}{dt}\left[ \begin{matrix}
m_x \\ m_y
\end{matrix} \right] 
= 
\left[ \begin{matrix}
\alpha & \beta \\
\gamma/\epsilon & \zeta/\epsilon
\end{matrix}  \right]
\left[ \begin{matrix}
m_x \\ m_y
\end{matrix} \right],
\qquad 
\textcolor{black}{
\left[ \begin{matrix}
m_x(0) \\ m_y(0)
\end{matrix} \right]
=
\left[ \begin{matrix}
m_{x,0} \\ m_{y,0}
\end{matrix} \right].
}
\label{OU_mean_full}
\end{equation}
\noindent
The evolution of the covariance of \eqref{slow_fast_OU_full} is given by the linear ODE $\dot \Sigma = B_{\Sigma} \Sigma + b_{\Sigma}$ with $\Sigma(0) = \Sigma_0$:
\begin{equation}
\frac{d}{dt}
\left[ \begin{matrix}
\Sigma_x \\ \Sigma_{xy} \\ \Sigma_y 
\end{matrix} \right] 
= 
\left[ 
\begin{array}{c | c c}
2\alpha &   2 \beta & 0 \\ \hline
\gamma/\epsilon & \alpha + \zeta/\epsilon & \beta \\
0 & 2\gamma/\epsilon & 2 \zeta/ \epsilon
\end{array}  \right]
\left[ \begin{matrix}
\Sigma_x \\ \Sigma_{xy} \\ \Sigma_y 
\end{matrix} \right]
+
\left[ \begin{matrix}
\sigma^2 \\ 0 \\ \sigma^2/\epsilon
\end{matrix} \right] ,
\qquad 
\textcolor{black}{
\left[ \begin{matrix}
\Sigma_x(0) \\ \Sigma_{xy}(0) \\ \Sigma_y(0) 
\end{matrix} \right] 
= 
\left[ \begin{matrix}
\Sigma_{x,0} \\ \Sigma_{xy,0} \\ \Sigma_{y,0} 
\end{matrix} \right],
}
\label{OU_variance_full}
\end{equation}
where we name 
$
B_{\Sigma}(\epsilon)
= \left[ 
\begin{array}{c | c}
2\alpha & p_{\Sigma}^T \\ \hline
q_{\Sigma}/\epsilon & \,-A_{\Sigma}(\epsilon)/\epsilon
\end{array}  \right],
$
with $A_\Sigma(\epsilon) = - 
\left[ 
\begin{matrix}
\alpha\textcolor{black}{\epsilon} + \zeta &\, \beta \textcolor{black}{\epsilon}	\\
2\gamma & 2\zeta 
\end{matrix}  
\right].$
To ensure stability of the fast dynamics, we assume that the parameters in \eqref{slow_fast_OU_full} are chosen such that the real part of the eigenvalues of the matrix $A_{\Sigma}(\epsilon)$ are all positive.
This condition is satisfied for instance for any $\alpha, \beta \in \mathbb{R}$ if $\zeta$ and $\gamma$ are sufficiently small.

\noindent
\textbf{\textit{Statistical moments of reduced model.}}
The evolution of the mean of $X$ in \eqref{extra_clarity_SDE_slow} is given by
\begin{equation}
\frac{dm_X}{dt} = \left( \alpha - \frac{\beta \gamma}{\zeta} \right) m_X,  
\qquad 
\textcolor{black}{
m_X(0) = m_{x,0}.
}
\label{OU_mean_reduced}
\end{equation}
The evolution of the variance $\Sigma_X$ of the reduced system \eqref{extra_clarity_SDE_slow} is given by
\begin{equation}
\frac{d\Sigma_X}{dt} = \lambda_{\Sigma} \Sigma_X  + \Sigma_{\Sigma}^2 = 2 \left( \alpha - \frac{\beta \gamma}{\zeta} \right)  \Sigma_X + \sigma^2,
\qquad
\textcolor{black}{
\Sigma_X(0) = \Sigma_{x,0}.
}
\label{OU_variance_reduced}
\end{equation} 

%
%
\noindent
We now briefly discuss a different perspective on the evolution ODE for the slow variance \eqref{OU_variance_reduced}.
We will use this insight later in the proof of theorem.
\begin{remark}[Alternative derivation of reduced model for evolution of slow variance]
The evolution of the slow variance \eqref{OU_variance_reduced} describes the second moment of the averaged SDE \eqref{extra_clarity_SDE_slow}.
There exists, however, a connection between 
(i) the ODE \eqref{OU_variance_reduced} based on the reduction technique for SDEs in equation \eqref{reduced_stochastic}, and 
(ii) a similar ODE derived through the reduction technique for multiscale ODEs in equation \eqref{definition_lambda} for linear multiscale ODEs.
This second reduction technique is used in the context of the micro-macro Parareal algorithm for linear multiscale ODEs in  \parencite{Legoll2013}.
This latter ODE is
\begin{equation}
\frac{dm_X}{dt} = \lambda_{\Sigma,\epsilon}(\epsilon),
\end{equation}
where 
\begin{equation}
\begin{aligned}
\lambda_{\Sigma,\epsilon}(\epsilon) 
= 2 \alpha + p_{\Sigma}^T A_{\Sigma}(\epsilon)^{-1} q_{\Sigma}
= 2 \alpha -
\frac{2 \beta \gamma \zeta}{(\epsilon \alpha + \zeta)\zeta - \gamma \beta \epsilon}.
\end{aligned}
\label{OU_lambda_reduced_model_exact}
\end{equation}
Now we can interpret the decay parameter $\lambda_{\Sigma}$ from the averaged model \eqref{OU_variance_reduced} as a limit case: 
\begin{equation}
\begin{aligned}
\lambda_{\Sigma}
= 2 \left( \alpha - \frac{\beta \gamma}{\zeta} \right)
= \lim_{\epsilon \rightarrow 0} \lambda_{\Sigma,\epsilon}(\epsilon) 
\end{aligned}
\end{equation}
In words, the variance of the slow SDE \eqref{extra_clarity_SDE_slow} does, in general, not obey the same evolution equations as when the reduction technique for ODEs in equation \eqref{definition_lambda} is applied on the moment ODE \eqref{OU_variance_full} of the multiscale SDE \eqref{slow_fast_OU_full}. These models, however, are closely related.
\label{remark_covariance_reduced_model}
\end{remark}

\subsection{Using micro-macro Parareal for the model problem}
\label{subsection_Parareal_for_OU}
In this work we are interested in the moments of the multiscale SDE solution paths (corresponding to weak convergence). 
We thus select the micro variable, describing the first two moments of its solution as 
$\textcolor{black}{u} = \left( \begin{matrix} m_x & m_y & \Sigma_x & \Sigma_{xy} & \Sigma_y \end{matrix} \right)$.
The macro variable is defined as 
$\textcolor{black}{U} = \left( \begin{matrix} m_x & \Sigma_x \end{matrix} \right)$.
The fine propagator $\mathcal{F}_{\Delta t}$ is the weak solution to the SDE \eqref{slow_fast_OU_full}, which we model via its statistical moment equations \eqref{OU_mean_full} and \eqref{OU_variance_full}. The coarse propagator $\mathcal{C}_{\Delta t}$ simulates the reduced system \eqref{reduced_stochastic}, or, equivalently, the scalar moment ODEs \eqref{OU_variance_full} and \eqref{OU_variance_reduced}.
As a quantity of interest, we are interested in the evolution of its first moments (i.e. mean and variance).

The coupling operators for the SDE problem, namely restriction $\mathcal{R}_{\mathrm{SDE}}$, matching $\mathcal{M}_{\mathrm{SDE}}$ and lifting $\mathcal{L}_{\mathrm{SDE}}$, are defined as 
\begin{equation}
\begin{aligned}
\mathcal{R}_\mathrm{SDE} \left(  
\left( \begin{matrix}
{m_x} &  m_y  & {\Sigma_x} & \Sigma_{xy} & \Sigma_y
\end{matrix} 
\right) \right) 
&= 
\left( \begin{matrix}
{m_x} & {\Sigma_x} \end{matrix} \right)
\\
\mathcal{M}\textcolor{black}{_\mathrm{SDE}} \left( 
\left( \begin{matrix}
{M_X} & {S_X}
\end{matrix} \right), 
\left[ \begin{matrix}
m_x & m_q & \Sigma_x & \Sigma_{xy} & \Sigma_y 
\end{matrix} \right)
 \right) 
&= 
 \left( \begin{matrix}
{M_X} & m_y & {S_X} & \Sigma_{xy} & \Sigma_y 
\end{matrix} \right),
\\
\mathcal{L}\textcolor{black}{_\mathrm{SDE}} 
\left( 
\left( \begin{matrix}
{M_X} & {S_X}
\end{matrix} 
\right) \right)  
&= 
\left( \begin{matrix}
{M_X} & m_{y,0} & {S_X} & \Sigma_{xy,0} & \Sigma_{y,0} 
\end{matrix} \right).
\end{aligned}
\label{coupling_operators_SDE_problem}
\end{equation}
We also define $\mathcal R_{\mathrm{SDE}}^{\perp}(X,y) = y$.
In words, the restriction operator $\mathcal{R}_\mathrm{SDE}$ extracts the slow mean and slow variance from the state variable that contains all micro means and variances.
The matching operator $\mathcal{M}_\mathrm{SDE}$ replaces the slow mean and the slow variance in the micro variable with desired macro variables.
The lifting operator  $\mathcal{L}_\mathrm{SDE}$ initializes the moments of the fast variable to their initial value. 

\subsection{Overview of assumptions on model parameters}
\label{section_overview_assumptions}
We require these assumptions on the parameters of the Ornstein-Uhlenbeck SDE:
\begin{itemize} 
\item We assume that there exists a value 
$\mu_{\Sigma,-} \leq $ such that the eigenvalues $\mu_{\Sigma}(\epsilon)$ of the matrix $A_{\Sigma}(\epsilon)$ are all positive and bounded: 
$0 < \mu_{\Sigma,-}
\leq 
\operatorname{Re}(\mu_{\Sigma}(\epsilon))$ .
Moreover  
$0 < \mu_{\Sigma,-} \leq \mu_{\Sigma,\epsilon,-}$ for all $\epsilon \in (0,1)$. 
Thus, as a consequence, $A_{\Sigma}(\epsilon)$ is invertible for all $\epsilon \in (0,1)$;


\item for all values $\epsilon \in (0,1)$, 
$\lambda_{\Sigma}$ satisfies $\lambda_{\Sigma,-} \leq \lambda_{\Sigma} \leq \lambda_{\Sigma,+}$ for some $\lambda_{\Sigma,-}$ and $\lambda_{\Sigma,+}$ independent of $\epsilon$. We also assume that $\lambda_{\Sigma} \neq 0$ and that $\lambda_{\Sigma, \epsilon} \neq 0$ for all $\epsilon \in (0,1)$;

\end{itemize}
If these conditions are satisfied, then the matrix $B_{\Sigma}$, defined in equation \eqref{OU_variance_full}, is invertible for all $\epsilon \in (0,1)$.
Indeed, in that case an explicit expression for the inverse of $B_{\Sigma}$ is given in lemma \ref{lemma_inverse_B_SIGMA}.

\begin{lemma}[Inverse of $B_{\Sigma}$]
\label{lemma_inverse_B_SIGMA}
If $\lambda_{\Sigma} \neq 0$ (defined in equation \eqref{OU_lambda_reduced_model_exact}), then the inverse of the matrix $B_{\Sigma}$ is given the next formula. 
For notational brevity, we omit the dependence of $A_{\Sigma}$ on $\epsilon$.
\begin{equation}
\begin{aligned}
B_{\Sigma}^{-1}
&= \left[ 
\begin{array}{c | c}
2\alpha & p_{\Sigma}^T \\ \hline
q_{\Sigma}/\epsilon & \,-A_{\Sigma}/\epsilon
\end{array}  \right]^{-1} \\
&= 
\lambda_{\Sigma,\epsilon}^{-1}
\left[ 
\begin{array}{c | c}
1 &  \epsilon p_{\Sigma}^T A_{\Sigma}^{-1} \\ 
\hline 
A_{\Sigma}^{-1} q_{\Sigma}
& 
-\epsilon A_{\Sigma}^{-1} \lambda_{\Sigma,\epsilon} + \epsilon A_{\Sigma}^{-1} q_{\Sigma} p_{\Sigma}^T A_{\Sigma}^{-1} 
\end{array}  \right]
\end{aligned}
\label{equation_Schur_complement}
\end{equation}
where $\lambda_{\Sigma,\epsilon} = 2\alpha + p_{\Sigma}^T A_{\Sigma}^{-1} q_{\Sigma}$ (see equation \eqref{OU_lambda_reduced_model_exact}).
\end{lemma}
\begin{proof}
The result follows from substituting the submatrices of $B_{\Sigma}$ in the inversion formula for a block matrix, see, e.g., \parencite[Proposition 2.8.7]{bernstein_matrix_2005}. 
Equation \eqref{equation_Schur_complement} is only valid if $A_{\Sigma}$ is nonsingular.
This condition is satisfied since $A_{\Sigma}(\epsilon)$ is assumed to have strictly positive eigenvalues (see section \ref{section_overview_assumptions}).
\end{proof}


\section{Convergence of micro-macro Parareal for the multiscale Ornstein-Uhlenbeck SDE model problem}
\label{section_micro_macro_Parareal_for_OU}
\textbf{\textit{Convergence of the first moment $(m_x, \, m_y)$.}} 
The moment equations \eqref{OU_mean_full} and \eqref{OU_mean_reduced}, describing the evolution of the first moment of the multiscale SDE and its reduced model, obey the structure of the multiscale system \eqref{test_system_Samaey_2013} \textcolor{black}{and its reduced model \eqref{definition_lambda}}. 
Therefore, we can directly apply lemma \ref{lemma_no_lifting}.

\noindent
\textbf{\textit{Convergence of the covariance $(\Sigma_x, \, \Sigma_{xy}, \, \Sigma_y)$.}} 
For the covariance it is not possible to directly apply results from \parencite{Legoll2013}, on which we base our analysis. 
Indeed, it is not directly clear whether or not the evolution ODE of the multiscale covariance \eqref{OU_variance_full} 
and its reduced model \eqref{OU_variance_reduced}
satisfy the same required properties as 
those that are used in the convergence analysis in \parencite{Legoll2013} (namely the ODE \eqref{test_system_Samaey_2013} and the reduced model \eqref{definition_lambda}). 
We summarize the differences with the theory in \parencite{Legoll2013}:
\begin{enumerate}[(i)]
\item here the ODEs are not homogeneous;

\item the reduced (moment) ODE \eqref{OU_variance_reduced} is not derived from the multiscale (moment) ODE \eqref{OU_variance_full} using the reduction technique \eqref{definition_lambda}, but instead it is defined via applying a reduction technique on the underlying SDE (see also remark \ref{remark_covariance_reduced_model});

\item the coefficients of the submatrix $A_{\Sigma}(\epsilon)$ depend on the parameter $\epsilon$.
\end{enumerate}
We stick closely to the analysis in \parencite{Legoll2013}. 

The sequel is organized as follows.
In section \ref{section_homogenised_covariance} we study the model error of the homogeneous part of the ODEs \eqref{OU_variance_reduced} and its reduced model \eqref{OU_variance_full} (this corresponds to $\sigma = 0$).
In section \ref{section_properties_nonhomog_covariance} we study the model error of ODEs \eqref{OU_variance_reduced} and its reduced model \eqref{OU_variance_full} in the fully nonhomogeneous case.
In section \ref{section_Parareal_stuff} we present our main lemma, namely a convergence analysis of micro-macro Parareal for the evolution of the covariance (described by an inhomogeneous constant-coefficient system of ODEs).

\subsection{Properties of the homogeneous part of the ODE describing the covariance}
\label{section_homogenised_covariance}
In this section we study some properties of the homogeneous part of the ODE describing the multiscale evolution of the covariance, and the reduced model for the variance of the slow variable.
The key difference with the existing theory is that some coefficients of the submatrix $A_{\Sigma}(\epsilon)$ depend on $\epsilon$.
We first give a property of the matrix $A_{\Sigma}(\epsilon)$ that will be required for later use. 

\begin{property}[Property of the matrix $A_{\Sigma}$]
There exist constants $C>0$ and $\mu > 0$ such that, for all $\epsilon \in (0,1)$ and for all $t \geq 0$,
\begin{equation}
\norm{e^{-A_{\Sigma}(\epsilon)t}} 
\leq 
C e^{- \frac{\mu_{\Sigma,-}}{2} t}
\qquad \mathrm{and} \qquad
\norm{A_{\Sigma}^{-1}(\epsilon)} \leq \frac{C}{\mu_{\Sigma,-}}
\label{equation_norm_A_Sigma_inverse}
\end{equation}
\label{lemma_property_matrix_A_Sigma}
\end{property}

\begin{proof}
Since we assume that the eigenvalues of $A_{\Sigma}$ are all negative, for a given $\epsilon$, we can use \parencite[Lemma 15]{Legoll2013}: for all $\epsilon \in (0,1)$ it holds that
\begin{equation}
\norm{e^{-A_{\Sigma}(\epsilon)t}} 
\leq 
C e^{- \frac{\mu_{\Sigma,\epsilon,-}}{2} t} 
\qquad \mathrm{and} \qquad
\norm{A_{\Sigma}^{-1}(\epsilon)} 
\leq \frac{C}{\mu_{\Sigma,\epsilon,-}}.
\end{equation}

Since we assume that
$0 < \mu_{\Sigma,-} \leq \mu_{\Sigma,\epsilon,-}$
for all $\epsilon \in (0,1)$, it holds that
$e^{-\mu_{\Sigma,\epsilon,-}} 
\leq 
e^{-\mu_{\Sigma,-}} $ 
and that $\frac{C}{\mu_{\Sigma,\epsilon,-}} \leq \frac{C}{\mu_{\Sigma,-}}$. 
This ends the proof.
\end{proof}

Now we are in a position to study how the reduced model approximates the variance of the slow component of the multiscale system, cfr. lemma \ref{lemma_properties_multiscale_system_2013}. 
The proof is equivalent to the proof of \parencite[Lemma 2 and Corrolary 3]{Legoll2013}.
However, extra care is required because a different reduction technique is used to arrive at a reduced model (see also remark \ref{remark_covariance_reduced_model}).

\begin{lemma}[Properties of the homogeneous parts of the multiscale equation \eqref{OU_variance_full} \textcolor{black}{and its reduced model \eqref{OU_variance_reduced}}]
Consider the homogeneous part of system \eqref{OU_variance_full} with initial condition $\left[ \Sigma_x(0), \Sigma_{xy}, \Sigma_y \right] = 
\left[ \Sigma_{x,0}, \Sigma_{xy,0}, \Sigma_{y,0} \right]$, 
and its reduced model \eqref{OU_variance_reduced} 
(this homogeneous part corresponds to $\sigma = 0$), with initial condition $\Sigma_X(0) = \Sigma_{x,0}$.
Then, there exists $\epsilon_0 \in (0,1)$ and $C>0$, independent of $\epsilon$, such that for all $\epsilon < \epsilon_0$,
\begin{equation}
\sup_{t \in [0,T]} |\Sigma_x(t) - \Sigma_X(t)| 
\leq C \epsilon \left(
|\Sigma_{x,0}| + \norm{ 
\mathcal{R}_\mathrm{SDE}^{\perp} \left( \Sigma(0) \right)
- A_{\Sigma}^{-1}q_{\Sigma}\Sigma_{x,0} }
\right).
\label{theorem_1}
\end{equation}

Assuming that the eigenvalues $\mu_{\Sigma,i}(\epsilon)$ of the matrix $A_{\Sigma}(\epsilon)$ (see \eqref{OU_variance_full}) are all positive,
there exists $\epsilon_0 \in (0,1)$ and $C >0$, independent of $\epsilon$, such that
\begin{equation}
\begin{aligned}
\sup_{t \in [0,T]} 
|\Sigma_x(t)| 
&\leq
 C \left( |\Sigma_{x,0}| + \epsilon \norm{
\mathcal R_{\mathrm{SDE}}^{\perp} \left( \Sigma(0) \right) 
} \right ), 
\\
\sup_{t \in [t^{\mathrm{BL}}_{\Sigma, \epsilon},T]} 
\norm{ 
\mathcal R_{\mathrm{SDE}}^{\perp} \left( \Sigma \right) } 
&\leq
C \left( |\Sigma_{x,0}| + \epsilon \norm{
\mathcal R_{\mathrm{SDE}}^{\perp} \left( \Sigma(0) \right)
} 
\right)
\end{aligned}
\label{equation_properties_homogenised_system}
\end{equation}
\label{lemma_my_2013_eq_2_14}
where $t^{\mathrm{BL}}_{\Sigma, \epsilon}$ is the length of a boundary layer in time of the order of $\epsilon$:
\begin{equation*}
t^{\mathrm{BL}}_{\Sigma \epsilon} = \frac{2 \epsilon}{\mu_{\Sigma,-}} \ln\left( \frac{1}{\epsilon} \right).
\end{equation*}
\label{lemma_my_2013_eq_2_8}
\end{lemma}

\begin{proof}
We first prove equation \eqref{theorem_1}.
We start by observing that
\begin{equation}
\begin{aligned}
&\sup_{t \in [0,T]}|\Sigma_x(t) -  e^{\lambda_{\Sigma}t}\Sigma_{x,0}|  \\
&\quad \leq 
\textcolor{black}{\sup_{t \in [0,T]}|\Sigma_x(t) -  e^{\lambda_{\Sigma,\epsilon} t}\Sigma_{x,0}| 
+
| e^{\lambda_{\Sigma,\epsilon} t} \Sigma_{x,0} - e^{\lambda_{\Sigma} t} \Sigma_{x,0}| } 
\end{aligned}
\label{equation_upper_bound_homog_lemma}
\end{equation}
We bound the first term of equation \eqref{equation_upper_bound_homog_lemma} using lemma \ref{lemma_properties_multiscale_system_2013} (equation \eqref{my_2013_eq_2_8}): there exists $\epsilon_0 \in (0,1)$ and a constant $C$, independent of $\epsilon$, such that for all $\epsilon < \epsilon_0$
\begin{equation}
\sup_{t \in [0,T]} 
|\Sigma_x(t) -  e^{\lambda_{\Sigma,\epsilon} t} \Sigma_{x,0}| \leq C \epsilon (|\Sigma_{x,0}| +  \| \Sigma_{z,0} \|).
\label{cool_equation}
\end{equation}
To bound the second term of equation \eqref{equation_upper_bound_homog_lemma}, we first define $\Delta \lambda_{\Sigma}  
= 
\lambda_{\Sigma}  - \lambda_{\Sigma,\epsilon}$, which satisfies
\begin{equation}
\begin{aligned}
\Delta \lambda_{\Sigma}  
&= 
2\left( \alpha - \frac{\beta \gamma}{\zeta} \right)
- \left(2 \alpha - \frac{2 \beta \gamma \zeta}{(\epsilon \alpha + \zeta)\zeta - \gamma \beta \epsilon} \right) \\ 
&= 
2 \beta \gamma \left( 
\frac{-(\epsilon \alpha + \zeta)\zeta + \gamma \beta \epsilon + \zeta^2}{\zeta \left[ (\epsilon \alpha + \zeta)\zeta - \gamma \beta \epsilon)\right]}  \right) \\
&= 
\frac{2\beta \gamma }{\zeta} \left[ \frac{ - \epsilon(\alpha \zeta - \gamma \beta)}{\epsilon(\alpha \zeta - \gamma \beta) + \zeta^2} \right] \\
&= 
-\frac{2 \beta \gamma}{\zeta} \frac{A \epsilon}{A \epsilon + B} \\
\end{aligned} 
\label{equation_delta_lambda_sigma}
\end{equation}
where $A$ and $B$ are constants, independent of $\epsilon$.
Now there exists $\epsilon_0 \in (0,1)$ and a constant $C$, independent of $\epsilon$, such that for all $\epsilon < \epsilon_0$,
\begin{equation}
|\Delta \lambda_{\Sigma}| \leq  \textcolor{black}{C \epsilon}.
\label{bound_delta_lambda}
\end{equation}
There exists a $\epsilon_0 \in (0,1)$ and a constant $C$, independent of $\epsilon$, such that for all $\epsilon < \epsilon_0$
\begin{equation}
\sup_{t \in [0,T]}
\left| e^{\lambda_{\Sigma,\epsilon} t} \Sigma_{x,0} - e^{\lambda_{\Sigma} t} \Sigma_{x,0} \right|
=
\sup_{t \in [0,T]}
e^{\lambda_{\Sigma}t}
\left|
1 - 
e^{
\Delta \lambda_{\Sigma} t
}
\right| 
\Sigma_{x,0}
\leq C \epsilon
\label{equation_helper_second_term}
\end{equation}
Finally, combining equations \eqref{cool_equation} and \eqref{equation_helper_second_term} there exist constants $C$, $K$ and $L$, independent of $\epsilon$, such that
\begin{equation}
\begin{aligned}
\sup_{t \in [0,T]}|\Sigma_x(t) - \Sigma_{x,0} e^{\lambda_{\Sigma}t}|
&\quad \leq 
	C \epsilon (|\Sigma_{x,0}| +  \| \Sigma_{z,0} \|) 
	+ K \epsilon |\Sigma_{x,0}| \\
&\quad \leq L \epsilon (|\Sigma_{x,0}| +  \| \Sigma_{z,0} \|),
\end{aligned}
\end{equation}
This proves equation \eqref{theorem_1}.
The proof of \eqref{equation_properties_homogenised_system} is similar to \parencite[Proof of lemma 3]{Legoll2013}.
\end{proof}

\subsection{Properties of the non-homogeneous ODE describing the covariance}
\label{section_properties_nonhomog_covariance}
In this section we study how the full inhomogeneous reduced model approximates the slow component of the inhomogeneous multiscale model, cfr. equation \eqref{theorem_1}.

\begin{lemma}[Approximation property of the reduced model \eqref{OU_variance_reduced} to the multiscale equation \eqref{OU_variance_full}]
Assume that the matrix $B_{\Sigma}$ is invertible.
For the multiscale ODE \eqref{OU_variance_full} 
and its reduced model \eqref{OU_variance_reduced}, 
there exist $\epsilon_0 \in (0,1)$ and $C>0$ and $K>0$, independent of $\epsilon$, such that, for all $\epsilon < \epsilon_0$
\begin{equation}
\sup_{t \in [0,T]} 
| \Sigma_x(t) - \Sigma_X(t)
| \leq C \epsilon
\label{theorem_1_non_homogeneous}
\end{equation}
where $\Sigma_{x}$ solves equation \eqref{OU_variance_full} and $\Sigma_{X}$ solves equation \eqref{OU_variance_reduced}.
\label{lemma_nonhomogeneous_error}
\end{lemma}
\begin{proof}
We assume that the matrix $B_{\Sigma}$ is invertible (see section \ref{section_overview_assumptions}).
The solution to equation \eqref{OU_variance_full}, $\dot \Sigma = B_{\Sigma} \Sigma + b_{\Sigma}$, is given by
\begin{equation}
\begin{aligned}
\Sigma(t) 
&= e^{B_{\Sigma}t} \left( \Sigma(0) + B_{\Sigma}^{-1} b_{\Sigma} \right) - B_{\Sigma}^{-1} b_{\Sigma}^{-1} \\
&= e^{B_{\Sigma}t} \Sigma(0) + \left[ e^{B_{\Sigma}t} - I\right] B_{\Sigma}^{-1} b_{\Sigma}^{-1}
\end{aligned}
\end{equation}
The solution to equation \eqref{OU_variance_reduced}, $\dot \Sigma_{X} = \lambda_{\Sigma} \Sigma_X + \sigma^2$, is given by 
\begin{equation}
\begin{aligned}
\Sigma_{X}(t) 
&= e^{\lambda_{\Sigma} t} 
\left( \Sigma_{X}(0) + \frac{\sigma^2}{\lambda_{\Sigma}} \right) - \frac{\sigma^2}{\lambda_{\Sigma}} \\
&= e^{\lambda_{\Sigma} t} \Sigma_{X}(0) + \left[ e^{\lambda_{\Sigma} t} - 1 \right]  \frac{\sigma^2}{\lambda_{\Sigma}}
\end{aligned} 
\end{equation}
\noindent
Thus we have that 
\begin{equation}
\begin{aligned}
&\sup_{t \in [0,T]} |\Sigma_x(t) - \Sigma_X(t)| 
= \sup_{t \in [0,T]} 
| \mathcal R_{\mathrm{SDE}} (\Sigma(t)) - \Sigma_X(t)|  \\
&=   
\sup_{t \in [0,T]} 
\left | 
\mathcal R_{\mathrm{SDE}} \left( 
e^{B_{\Sigma}t} \Sigma(0) + \left[ e^{B_{\Sigma}t} - I\right] B_{\Sigma}^{-1} b_{\Sigma}^{-1}
\right)
- \left(
e^{\lambda_{\Sigma} t} \Sigma_{X}(0) + \left[ e^{\lambda_{\Sigma} t} - 1 \right]  \frac{\sigma^2}{\lambda_{\Sigma}}
\right) \right| \\
&\leq 
\sup_{t \in [0,T]} 
\left| \mathcal R_{\mathrm{SDE}} \left( e^{B_{\Sigma}t} \Sigma(0) 
\right) - \Sigma_{x,0} e^{\lambda_{\Sigma} t} 
\right| 
\\
& \qquad + 
\sup_{t \in [0,T]}
\left| \left[ e^{\lambda_{\Sigma} t} - 1 \right]  \frac{\sigma^2}{\lambda_{\Sigma}}
-
\mathcal R_{\mathrm{SDE}}
\left[ e^{B_{\Sigma}t} - I\right] B_{\Sigma}^{-1} b_{\Sigma}^{-1} \right|,
\end{aligned}
\end{equation}
where we used the triangle inequality in the last step.
The first term can be bounded by (i) writing $\mathcal R_{\mathrm{SDE}} \left( 
e^{B_{\Sigma}t} \Sigma(0) \right) = \Sigma_x(t)$ and (ii) then using lemma \ref{lemma_my_2013_eq_2_8} (equation \eqref{theorem_1}), 
as well as the fact that, since $A_{\Sigma}$ is assumed to be invertible for all $\epsilon \in (O,1)$, all the terms in the right-hand side of \eqref{theorem_1} can be bounded.

\noindent
The second term can be bounded using lemma \ref{lemma_how_far_do_nonhomog_steady_states_lie_apart} (equation \eqref{equation_transient_model_error}).
\end{proof}

\subsection{Convergence of micro-macro Parareal for the inhomogeneous ODE system describing the covariance}
\label{section_Parareal_stuff}
\noindent
The preceding lemmas allow us to formulate our main result. 
\begin{lemma}
[Convergence of micro-macro Parareal for evolution of covariance]
We use micro-macro Parareal, defined in \eqref{mM_Parareal_0}-\eqref {mM_Parareal_other_k}, with fine and coarse propagators the exact solution of the inhomogeneous ODE \eqref{OU_variance_full} and the reduced ODE \eqref{OU_variance_reduced}, respectively.
We use the coupling operators $\mathcal{R}_{\mathrm{SDE}}$, $\mathcal{M}_{\mathrm{SDE}}$ and $\mathcal{L}_{\mathrm{SDE}}$ defined in equation \eqref{coupling_operators_SDE_problem}.
Let 
$e_{\Sigma,n}^k = \Sigma_n^k - \Sigma(t_n)$ be the micro error
and let
$E_{\Sigma_x,n}^k 
= \mathcal R_{\mathrm{SDE}} \Sigma^k_n - \mathcal R_{\mathrm{SDE}} \Sigma(t_n) 
= (\Sigma_x)_n^k - \Sigma_x(t_n)$ be the macro error. 
Then there exists $\epsilon_0 \in (0,1)$, that only depends on $\alpha$, $p_{\Sigma}$, $q_{\Sigma}$, $A_{\Sigma}$ and $T$, such that, for all $\epsilon < \epsilon_0$ and all $\Delta t > t_{\textcolor{black}{\Sigma},\epsilon}^{\mathrm{BL}}$, there exists a constant $C_k$, independent of $\epsilon$, such that for all $k \geq 0$:
\begin{equation}
\sup_{0 \leq n \leq N} |E_{\Sigma,n}^k| \leq C_k \epsilon^{\ceil{(k+1)/2}}
\end{equation}
\begin{equation}
\sup_{0 \leq n \leq N} \|e_{\Sigma,n}^k\| \leq C_k \epsilon^{\floor{(k+1)/2}}
\end{equation}
\label{theorem_convergence_mMParareal_variance}
\end{lemma}
\begin{proof}
In the zeroth iteration, 
lemma \ref{lemma_nonhomogeneous_error} can be used to bound the macro error: 
\begin{equation}
|E^0_{\Sigma,n}| \leq C \epsilon.
\end{equation}
For the micro error in the zeroth iteration, we have that 
\begin{equation}
\begin{aligned}
\sup_{0 \leq n \leq N}\|e^0_{\Sigma,n}\| 
&\leq 
\sup_{0 \leq n \leq N} \left( 
|E^0_{\Sigma,n}| + 
\norm{ \left[ \begin{matrix} \Sigma_{xy}(t_n) \\ \Sigma_{y}(t_n) \end{matrix} \right] - \left[ \begin{matrix} 
\Sigma_{xy,0} \\ \Sigma_{y,0} \end{matrix} \right] }
\right).
\end{aligned}
\end{equation}
The second term can be bounded using lemma \ref{lemma_distance_from_initial_condition}. 
This proves the bound in the zeroth iteration.
\newline
For all subsequent iterations, lemma \ref{convergence_Parareal_nonhomog_ODE} states that the inhomogeneity has no influence on the propagation of the error. 
The rest of the proof is analogous to the proof of lemma \ref{lemma_no_lifting}. 
\end{proof}

\section{Numerical experiments}
\label{section_numerical_experiments}

The test parameters for the numerical experiments are chosen to be: 
\begin{equation}
\left[ \begin{matrix}
\alpha & \beta \\
\gamma & \zeta
\end{matrix} \right] 
= 
\left[ \begin{matrix}
-1 &&& -1 \\
0.1 &&& -1
\end{matrix} \right], 
\qquad
\sigma = 0.5
\label{all_test_parameters}
\end{equation}
The time interval is chosen as $[0, T] = [0, 10]$, the number of time intervals $N=10$, and the initial value
$
\left( \begin{matrix}
m_{x,0} & m_{q,0} & \Sigma_{x,0} & \Sigma_{q,0} & \Sigma_{xq,0}
\end{matrix} \right) = \left( \begin{matrix} 100 &100 & 0 & 0 & 0 \end{matrix} \right)$.


\subsection{Choice of model parameters}
First we check whether the test parameters \eqref{all_test_parameters} satisfy (some of) the assumptions listed in section \ref{section_overview_assumptions}.
We numerically study how the eigenvalues of $A_{\Sigma}(\epsilon)$ and the scalar $\lambda_{\Sigma,\epsilon}$ change as $\epsilon$ varies.
This is illustrated in figure \ref{figure_eps_vs_model_parameters}.
For all values $0 \leq \epsilon \leq 1$ it holds that the real parts of the eigenvalues of $A_{\Sigma}$ are strictly postive, thus $A_{\Sigma}(\epsilon)$ is invertible. 

\begin{figure}[H]
\includegraphics[width=0.5\textwidth]{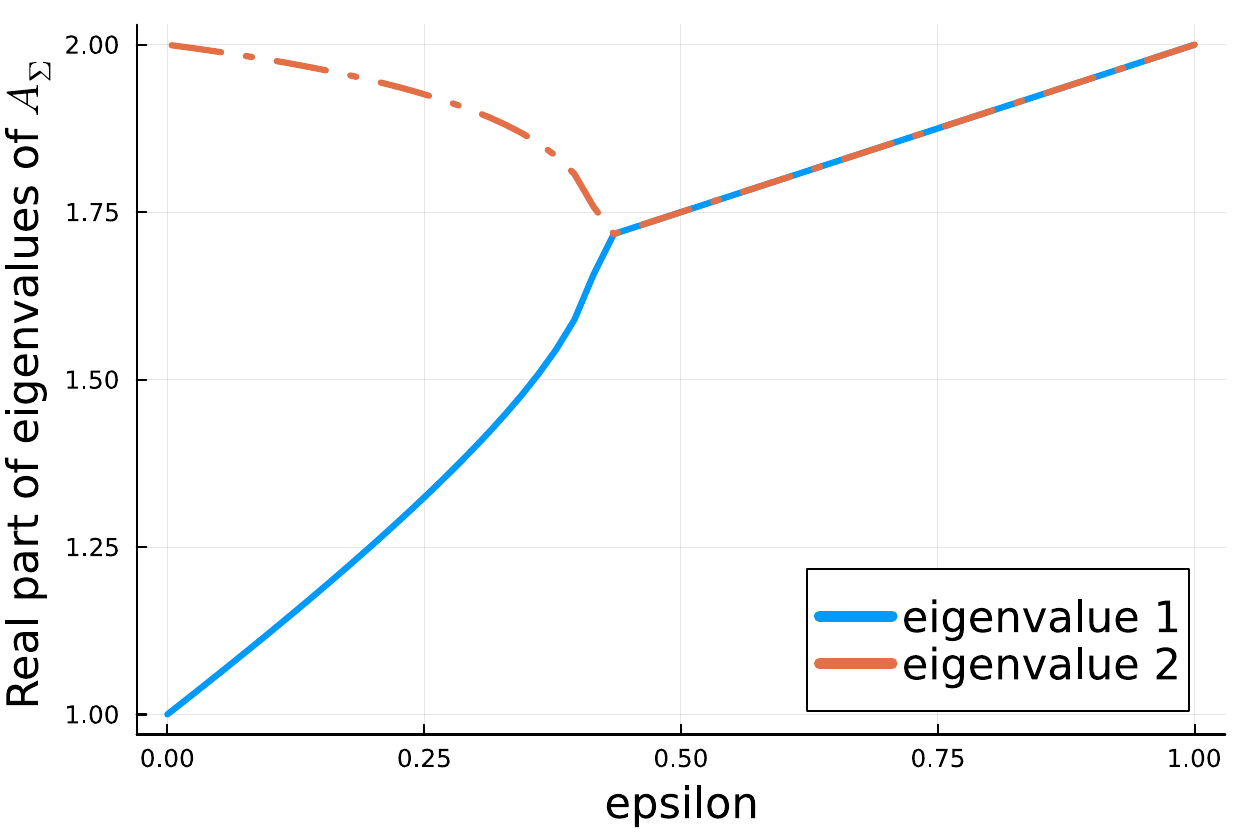}
\includegraphics[width=0.5\textwidth]{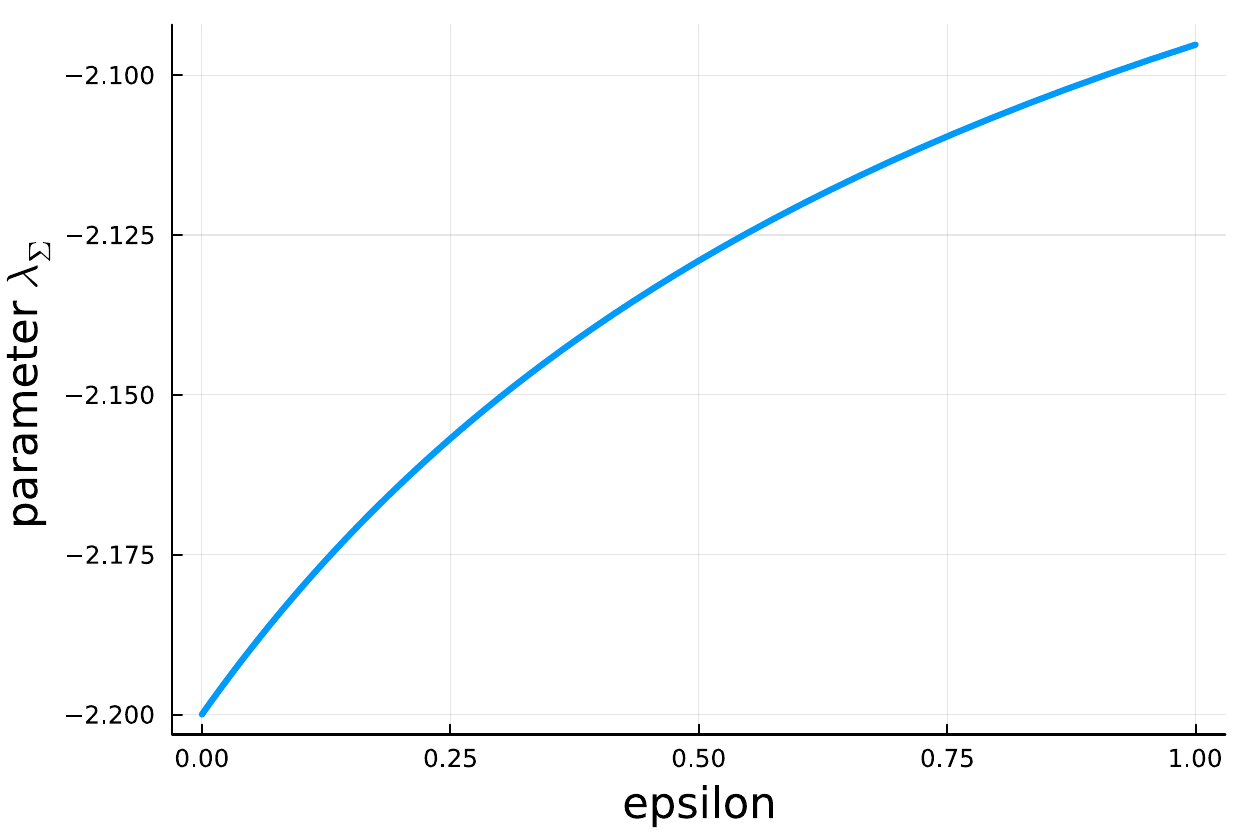}
\caption{As a function of $\epsilon$, we plotted (left) the real part of the eigenvalues of $A_{\Sigma}$ and (right) the value of $\mu_{\Sigma}$.}
\label{figure_eps_vs_model_parameters}
\end{figure}


\subsection{Convergence of micro-macro Parareal}
The numerical simulations are illustrated in figure \ref{numerical_experiments}. 
It is seen that the micro and macro errors \textcolor{black}{on} the mean follow the behavior given by lemma \ref{lemma_no_lifting}.
The errors on the variance follow the behavior as given by lemma \ref{theorem_convergence_mMParareal_variance}. Observe that micro-macro Parareal converges faster for computationally more expensive models (with small $\epsilon$), because the macro model becomes more accurate with respect to the micro model.

\begin{figure}
\includegraphics[width=0.5\textwidth]{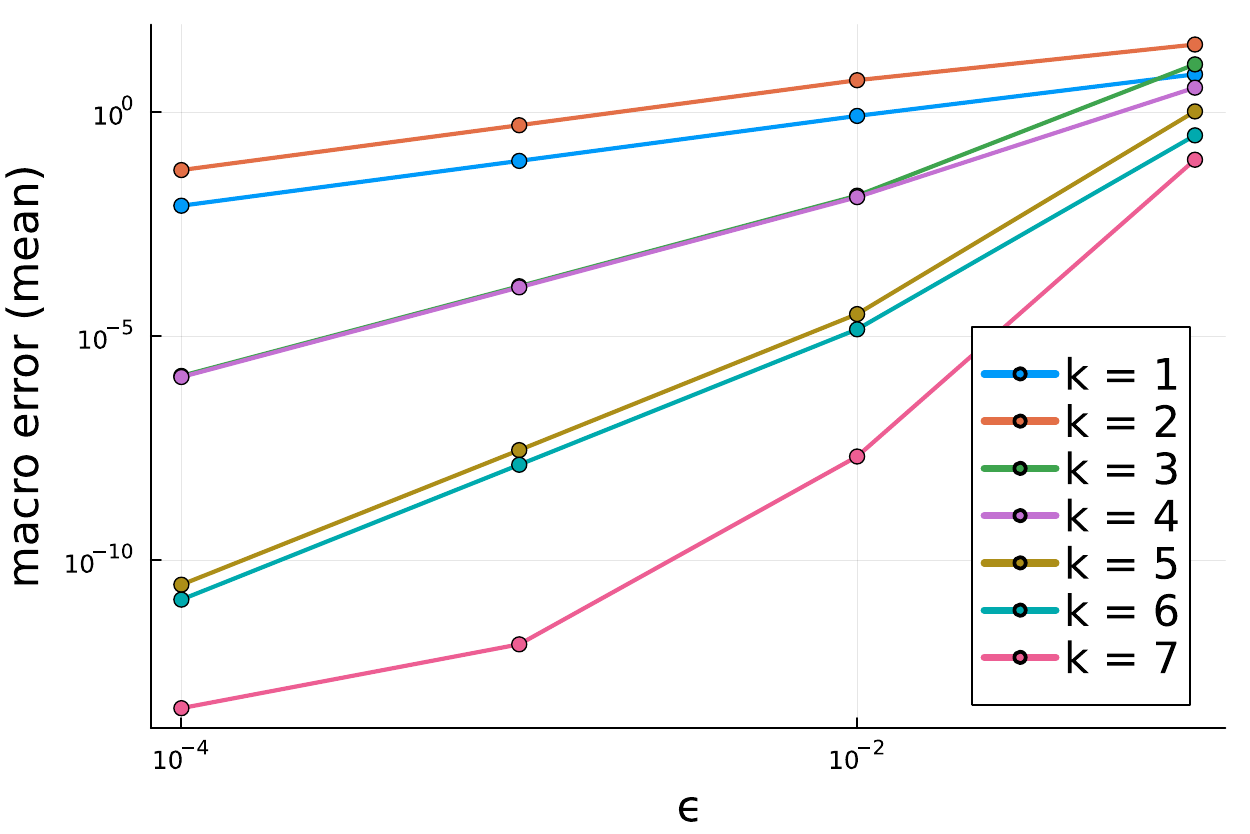}
\includegraphics[width=0.5\textwidth]{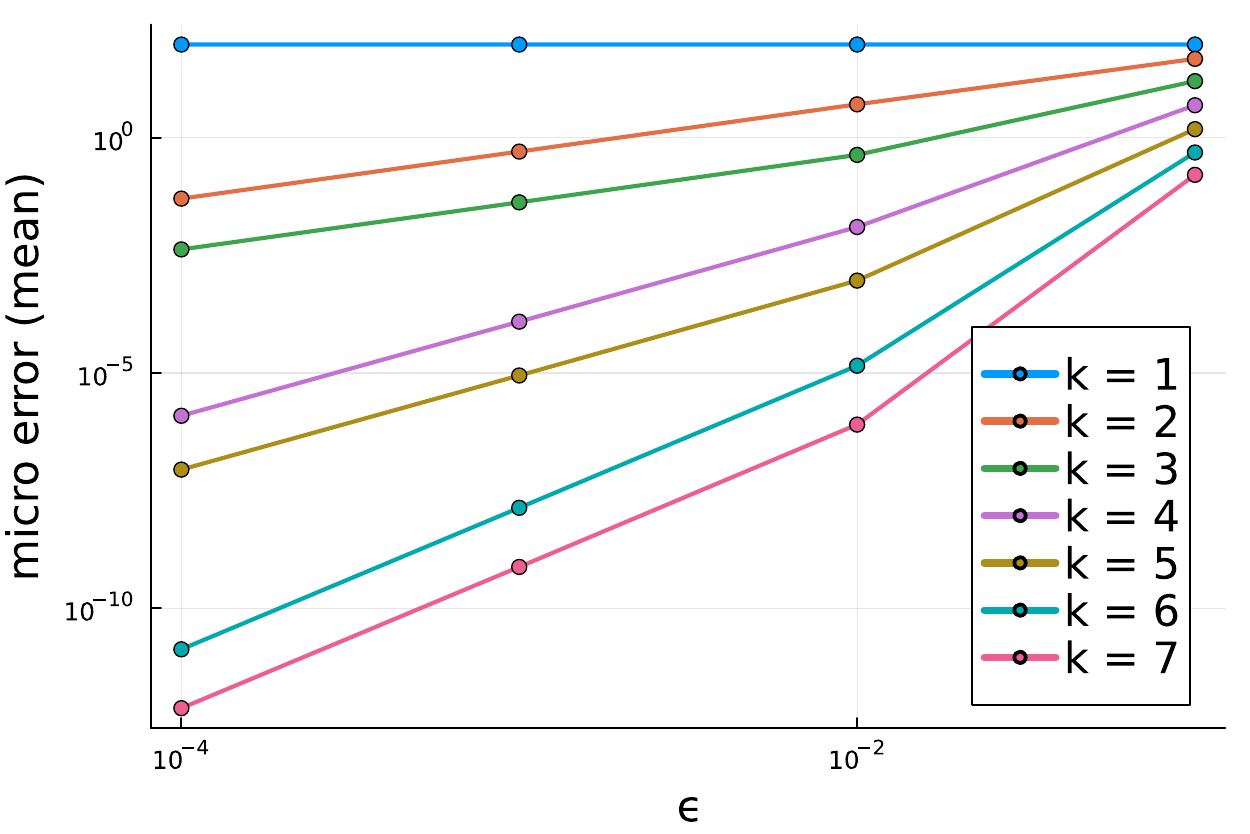}
\includegraphics[width=0.5\textwidth]{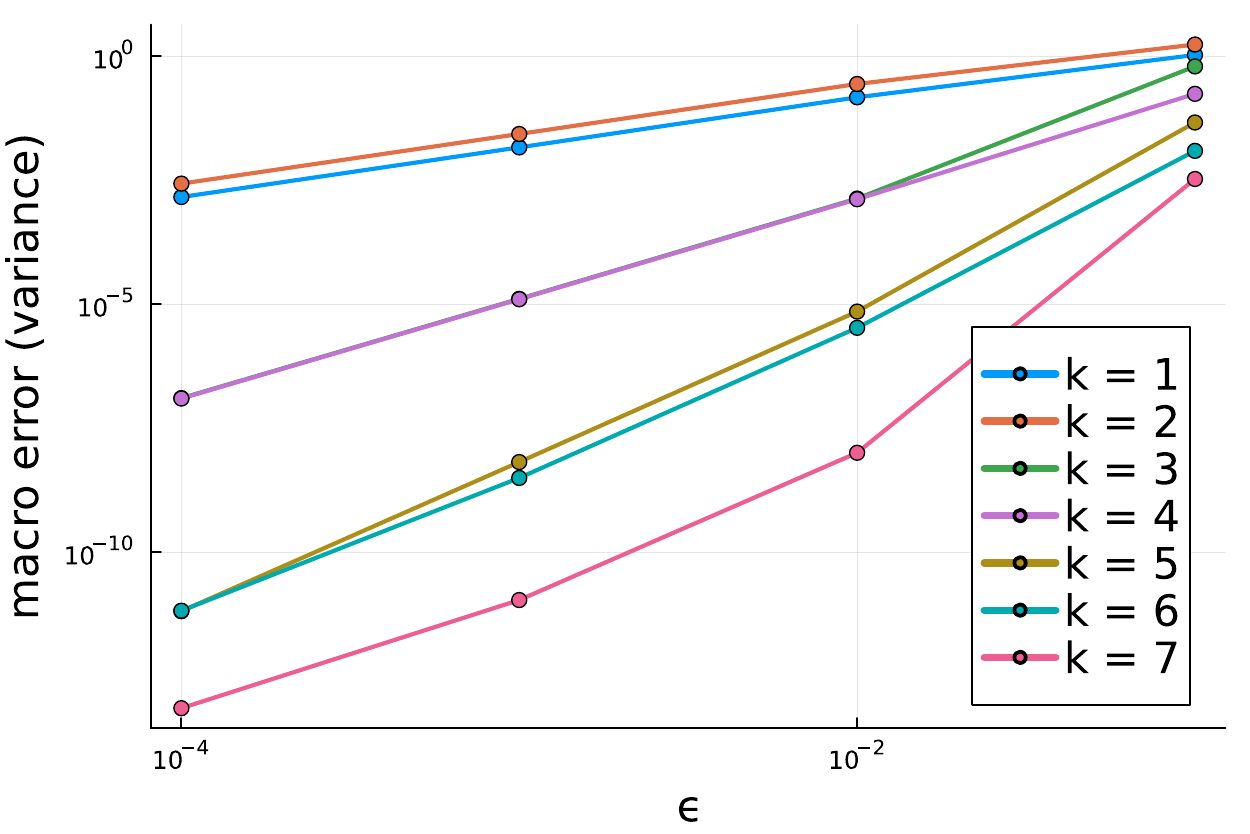}
\includegraphics[width=0.5\textwidth]{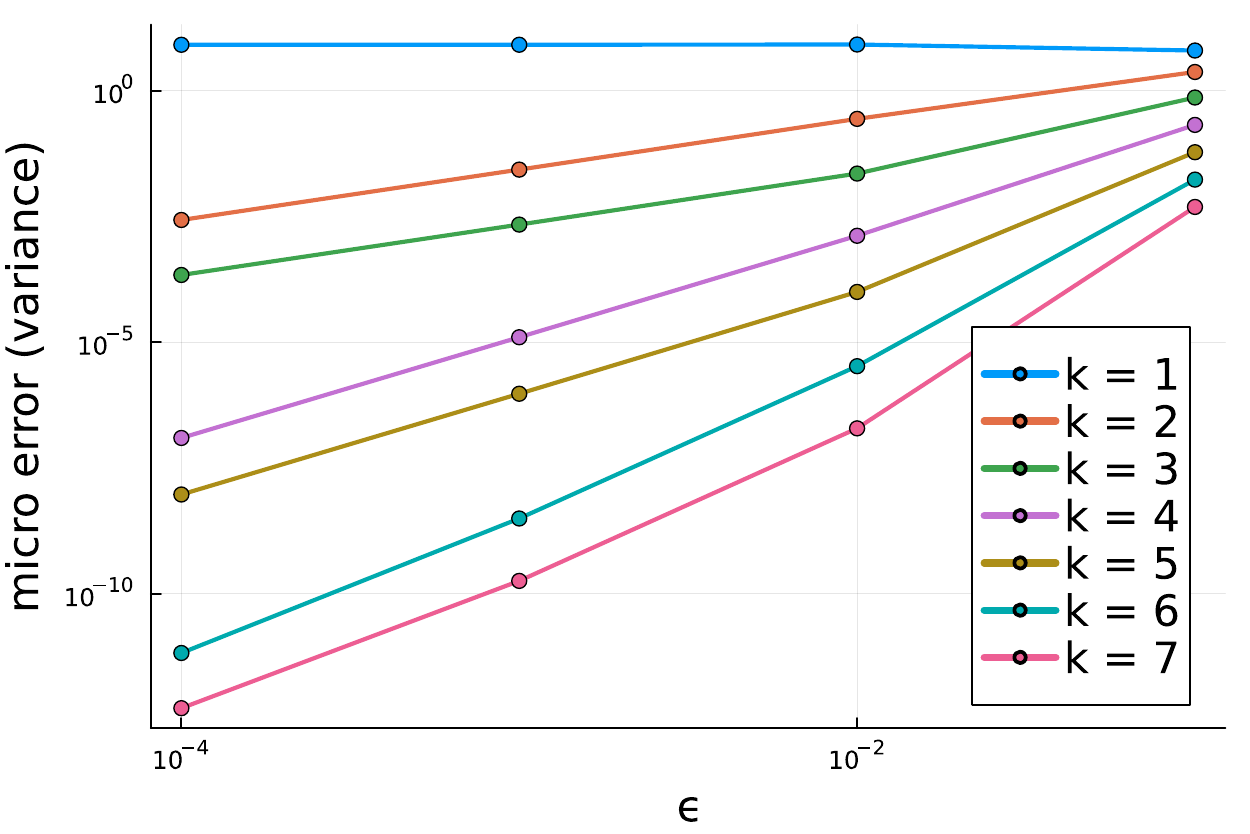}
\caption[Convergence of micro-macro Parareal for a multiscale Ornstein-Uhlenbeck SDE]{Error as function of time-scale separation parameter $\epsilon$. We used $\infty$-norm over time (only considering coarse discretization points) and the 2-norm for the micro error. 
\textbf{Top left}: macro error on mean, 
\textbf{Top right}: micro error on mean, 
\textbf{Bottom left}: macro error on variance,
\textbf{Bottom right}; micro error on variance. We used a numerical solver to discretize the moment equations \eqref{OU_mean_reduced} - \eqref{OU_variance_full} with a very stringent tolerance, so that the effect of numerical discretization errors can be neglected.}
\label{numerical_experiments}
\end{figure}

\section{Discussion and conclusion}
\-\hspace{0.3cm}
\textbf{Summary.}  
We presented a convergence analysis of the micro-macro Parareal algorithm applied on scale-separated Ornstein-Uhlenbeck SDEs. We analyzed its convergence behavior w.r.t. the time scale separation parameter $\epsilon$, using moment equations. The convergence of the first moment is closely related to the analysis in \parencite{Legoll2013}. For the covariance we presented some extensions to this theory.

\textbf{Limitations.}
While the analysis using moment equations quantifies the error on the mean and variance of the SDE solution, we cannot say anything about other quantities of interest, such as higher moments of the SDE solutions. 

Also, by using the moment equation (an ODE that we solved using very stringent tolerances), we exclusively looked at the model error, neglecting the discretization and statistical errors (in e.g. Monte Carlo simulations) that arise in the discretization of an SDE. 

\textbf{Open questions.}
It remains to be studied how the analysis generalizes to higher dimensions, for instance when the slow variable is multi-dimensional. 
Also, an extension of the convergence analysis could cover nonlinear SDEs, or linear SDEs for which there is a coupling between mean and variance in the moment ODEs.
Another open problem is the convergence analysis of the method w.r.t. the iteration number, instead of convergence w.r.t. the parameter $\epsilon$. This would be more useful in practice. 

\textbf{Software.}
The code that is used for the numerical experiments, is publicly available\footnote{\url{https://gitlab.kuleuven.be/numa/public/micro-macro Parareal-convergence-sde}}. We used the Julia language \parencite{bezanson_julia_2017} with the DifferentialEquations.jl package \parencite{rackauckas_differentialequationsjl_2017}.

\section{Proof of lemma \ref{convergence_Parareal_nonhomog_ODE}}
\label{proof_of_affine_Parareal}
\begin{proof}
The proof is similar to the derivation of equations (4.4) and (4.5) in \parencite{Legoll2013}, except that, here, the fine and coarse propagators are not linear but affine. 
With $B_{\mathcal F, n}$ and $B_{\mathcal C, n}$ we denote $B_{\mathcal F}(t_n)$ and $B_{\mathcal C}(t_n)$. 
To derive equation \eqref{correction_solution_Parareal} we proceed as follows. From equation \eqref{mM_Parareal_other_k}, it holds that
\begin{equation}
\begin{aligned}
U^{k+1}_n 
&= \mathcal R \left( 
A_{\mathcal F} u^k_{n-1} + B_{\mathcal F,n} \right) 
+ A_{\mathcal C} \left(  
U^{k+1}_{n-1} - U^k_{n-1} \right) \\
&= 
\mathcal R \left( A_{\mathcal F} u^k_{n-1} + B_{\mathcal F,n} \right)
+ A_{\mathcal C} \left( 
\mathcal R 
\left( A_{\mathcal F} u^k_{n-2} + B_{\mathcal F,n-1} 
\right)
+ A_{\mathcal C} \left(  
U^{k+1}_{n-2} - U^k_{n-2} \right) - U^k_{n-1} \right) \\ 
&= 
\hdots \\
&= 
\mathcal R 
\left( 
A_{\mathcal F} u^{k}_{n-1} + B_{\mathcal F,n}
\right)
+ \sum_{p=1}^{n-1} 
A_{\mathcal C}^{n-p} 
\left( \mathcal R 
\left( 
A_{\mathcal F} u^{k}_{p-1} + \mathcal R B_{\mathcal F,p}
\right)
- U^k_p \right)
\end{aligned}
\end{equation}
This proves equation \eqref{correction_solution_Parareal}.
Equation \eqref{correction_error_Parareal} can be obtained similarly to the derivation of equation (4.5) in \parencite{Legoll2013}, where the nonhomogeneous terms from the fine propagator cancel out.
\begin{equation}
\begin{aligned}
E^{k+1}_n 
&= 
U^{k+1}_n - \mathcal R u_n \\
&= 
U^{k+1}_n - \mathcal R 
\left( A_{\mathcal F} u_{n-1} + B_{\mathcal F,n} \right) \\
&= 
\mathcal R \left( 
A_{\mathcal F} u^{k}_{n-1} + B_{\mathcal F,n}
\right)
- \mathcal R \left( 
A_{\mathcal F} u_{n-1} + B_{\mathcal F,n} 
\right) 
+ \sum_{p=1}^{n-1} A_{\mathcal C}^{n-p} 
\left( \mathcal R 
\left( A_{\mathcal F} u^{k}_{p-1}  + B_{\mathcal F,p} \right)
- U^k_p \right) 
\\
&= 
\mathcal R A_{\mathcal F} e^{k}_{n-1}
+ \sum_{p=1}^{n-1} A_{\mathcal C}^{n-p} 
\left( \mathcal R 
\left( A_{\mathcal F} u^{k}_{p-1}  + B_{\mathcal F,p} 
\right)
- \mathcal R u_{p} + \mathcal R u_p
- U^k_p \right) \\
&= 
\mathcal R A_{\mathcal F} e^{k}_{n-1}
+ \sum_{p=1}^{n-1} A_{\mathcal C}^{n-p} 
\left( \mathcal R  \left( 
A_{\mathcal F} u^{k}_{p-1} + B_{\mathcal F,p} \right)
- \mathcal R \left( A_{\mathcal F} u_{p-1} + B_{\mathcal F,p} \right) + \mathcal R u_p 
- U^k_p \right) \\
&= \mathcal R A_{\mathcal F} e^{k}_{n-1}
+ \sum_{p=1}^{n-1} A_{\mathcal C}^{n-p} 
\left( \mathcal R A_{\mathcal F} e^{k}_{p-1} - E^k_p \right) \\
&= \sum_{p=1}^{n-1} A_{\mathcal C}^{n-p-1} 
\left( \mathcal R A_{\mathcal F} e^{k}_{p} 
- \sum_{p=1}^{n-1} A_{\mathcal C}^{n-p} E^k_p \right)
\end{aligned}
\end{equation}
From the second to the third line, we used equation \eqref{correction_solution_Parareal}.
In the fourth line, we added and subtracted $\mathcal R B_{\mathcal F}$ to the second factor of each term of the summation.
\end{proof}

\section{Auxiliary lemmas and proofs for lemma \ref{lemma_nonhomogeneous_error}}

\begin{lemma}[Property of the inhomogeneity in the multiscale model]
There exists a constants $C$, independent of $\epsilon$, such that for all $\epsilon \in (0,1)$
\begin{equation}
\norm{
\mathcal R^{\perp} _{\mathrm{SDE}}
\left( B_{\Sigma}^{-1} b_{\Sigma} \right)
} 
\leq C.
\end{equation}
\label{lemma_norm_inhomogeneity}
\end{lemma}

\begin{proof}
For brevity in notation, we omit the explicit dependence of $A_{\Sigma}(\epsilon)$ on $\epsilon$.
From the definition of $\mathcal R^{\perp}_{\mathrm{SDE}}$ in equation \eqref{coupling_operators_SDE_problem} and from equation \eqref{equation_Schur_complement} it follows that
\begin{equation}
\mathcal R^{\perp}_{\mathrm{SDE}} \left( B_{\Sigma}^{-1} b_{\Sigma} \right)
= 
\frac{\sigma^2 }{\lambda_{\Sigma, \epsilon}}
\left( 
A_{\Sigma}^{-1} q_{\Sigma} + 
\left[ - \lambda_{\Sigma} A_{\Sigma}^{-1} +  A_{\Sigma}^{-1} q_{\Sigma}  p_{\Sigma}^T  A_{\Sigma}^{-1} \right] 
\left[ \begin{matrix}
0 \\ 1
\end{matrix} \right] 
\right)
\end{equation}
Then, using the triangle inequality and using the assumption $\lambda_{\Sigma,-}, \lambda_{\Sigma,\epsilon}, \lambda_{\Sigma,+}$ (see section \ref{section_overview_assumptions}), we obtain
\begin{equation}
\begin{aligned}
\norm{\mathcal R^{\perp}_{\mathrm{SDE}} \left( B_{\Sigma}^{-1} b_{\Sigma} \right)} 
&\leq 
\frac{\sigma^2}{\lambda_{\Sigma,-}}
\left( \norm{A_{\Sigma}^{-1}} \norm{q_{\Sigma}}
+
\lambda_{\Sigma,+} \norm{A_{\Sigma}^{-1}}
+
\norm{A_{\Sigma}^{-1}} \norm{q_{\Sigma}} \norm{p_{\Sigma}^T} 
\norm{A_{\Sigma}^{-1}}
\right).
\end{aligned}
\end{equation}
Using equation \eqref{equation_norm_A_Sigma_inverse}, and the fact that there exist constants $K$ and $L$ and $M$, independent of $\epsilon$, 
such that $\norm{q_{\Sigma}} \leq K$ 
and $\norm{p_{\Sigma}^T} \leq L$
and $\norm{A^{-1}_{\Sigma}} \leq M$ (by equation \eqref{equation_norm_A_Sigma_inverse}),
we have
\begin{equation}
\norm{R_{\mathrm{SDE}}^{\perp} \left( B_{\Sigma}^{-1} b_{\Sigma} \right)} 
\leq 
C.
\end{equation}
\end{proof}

\begin{corollary}[Distance between steady state and initial condition]
There exists a constant $C$, independent of $\epsilon$, such that for all $\epsilon \in (0,1)$
\begin{equation}
\norm{
\mathcal R^{\perp}_{\mathrm{SDE}} 
\left(
\Sigma_0 - B_{\Sigma}^{-1}b_{\Sigma}
\right)} 
\leq 
\norm{ R^{\perp}_{\mathrm{SDE}} 
\left( \Sigma_0 \right) } + C
\end{equation}
\label{lemma_distance_from_initial_condition}
\end{corollary}
\begin{proof}
\begin{equation}
\begin{aligned}
\norm{
\mathcal R_{\mathrm{SDE}}^{\perp} \left((
\Sigma_0 - B_{\Sigma}^{-1}b_{\Sigma} \right)}
&\leq 
\norm{
\mathcal R_{\mathrm{SDE}}^{\perp} \left( \Sigma_0 \right) }
+
\norm{
\mathcal R_{\mathrm{SDE}}^{\perp} \left( B_{\Sigma}^{-1}b_{\Sigma} \right)} \\
&\leq
\norm{
\mathcal R_{\mathrm{SDE}}^{\perp} \left( \Sigma_0 \right) }
+
C
\end{aligned}
\end{equation}
where the second term is bounded using lemma \ref{lemma_norm_inhomogeneity}.
\end{proof}

\begin{lemma}[Steady-state (reduced) model error of the variance of the slow variable]
There exists $\epsilon_0 \in (0,1)$ and a constant $C$, independent of $\epsilon$, such that for all $\epsilon < \epsilon_0$
for all $\epsilon \in (0,1)$
\begin{equation}
\left| \frac{\sigma^2}{\lambda_{\Sigma}}
-
\mathcal R_{\mathrm{SDE}} 
\left( B_{\Sigma}^{-1} b_{\Sigma}^{-1}
\right)
 \right| 
\leq C \epsilon.
\label{blable}
\end{equation}
\label{easy_corrolary}
\end{lemma}
\begin{proof}
By the triangle inequality, it holds that 
\begin{equation}
\left| \frac{\sigma^2}{\lambda_{\Sigma}}
-
\mathcal R_{\mathrm{SDE}}
\left( B_{\Sigma}^{-1} b_{\Sigma}^{-1}
\right) \right|  
\leq 
\left| \frac{\sigma^2}{\lambda_{\Sigma, \epsilon}} - 
\frac{\sigma^2}{\lambda_{\Sigma}}
\right|
+ 
\left| \frac{\sigma^2}{\lambda_{\Sigma, \epsilon}}
-
\mathcal R_{\mathrm{SDE}} 
\left( B_{\Sigma}^{-1} b_{\Sigma}
\right) \right| 
\label{equation_with_two_terms}
\end{equation}
We can bound the first term of \eqref{equation_with_two_terms}, 
using equation \eqref{bound_delta_lambda}: there exists $\epsilon_0 \in (0,1)$ and a constant $C$, independent of $\epsilon$ such that for all $\epsilon < \epsilon_0$
\begin{equation}
\begin{aligned}
\left| \frac{\sigma^2}{\lambda_{\Sigma, \epsilon}} - 
\frac{\sigma^2}{\lambda_{\Sigma}}
\right|
&= 
\left| \frac{\sigma^2}{\lambda_{\Sigma, \epsilon} + \Delta \lambda_{\Sigma}}
-
\frac{\sigma^2}{\lambda_{\Sigma, \epsilon}}
\right|  \\
&= \sigma^2 \left| 
\frac{\Delta \lambda_{\Sigma}}{\lambda_{\Sigma, \epsilon}
(\lambda_{\Sigma, \epsilon} + \Delta \lambda_{\Sigma})}
\right|  \\
&\leq C \epsilon 
\end{aligned}
\label{bound_first_part}
\end{equation}

For the second term in equation \eqref{equation_with_two_terms}, 
we obtain from equation \eqref{equation_Schur_complement} and after substituting the elements of $b_{\Sigma}$, 
\begin{equation}
\begin{aligned}
\mathcal R_{\mathrm{SDE}}
\left( B_{\Sigma}^{-1} b_{\Sigma} \right) 
&= \lambda_{\Sigma,\epsilon}^{-1} \sigma^2 +
\epsilon \lambda_{\Sigma,\epsilon}^{-1}p_{\Sigma}^T A_{\Sigma}(\epsilon)^{-1} \left[ \begin{matrix}
0 \\ {\sigma^2}/{\epsilon}
\end{matrix} \right] \\
&= \lambda_{\Sigma,\epsilon}^{-1} \sigma^2 +
\epsilon \lambda_{\Sigma,\epsilon}^{-1}
 \left[ \begin{matrix}
2 \beta \zeta & -2 \beta^2 \epsilon
\end{matrix} \right]
 \left[ \begin{matrix}
0 \\ {\sigma^2}/{\epsilon}
\end{matrix} \right] \\
&= \lambda_{\Sigma,\epsilon}^{-1} \sigma^2 
-
2 \epsilon \lambda_{\Sigma,\epsilon}^{-1} \beta^2 \sigma^2 \\
\end{aligned}
\label{yow_eq}
\end{equation}
Thus we have that there exists a constant $C$, independent of $\epsilon$ such that
\begin{equation}
\left| \frac{\sigma^2}{\lambda_{\Sigma,\epsilon}} - 
\mathcal R_{\mathrm{SDE}}
\left( B_{\Sigma}^{-1} b_{\Sigma}
\right) \right| \leq C \epsilon
\label{bound_second_part}
\end{equation}
The combination of the bounds \eqref{bound_first_part} and \eqref{bound_second_part} leads to equation \eqref{blable}.
\end{proof}

\begin{lemma}[Transient model error of the variance of the slow variable]
There exists a constant $C$, independent of $\epsilon$, such that for all $\epsilon \in (0,1)$
\begin{equation}
\sup_{t \in [0,T]}
\left| \left[ e^{\lambda_{\Sigma} t} - 1 \right]  \frac{\sigma^2}{\lambda_{\Sigma}}
-
\mathcal R_{\mathrm{SDE}}
\left( 
\left[ e^{B_{\Sigma}t} - I\right] B_{\Sigma}^{-1} b_{\Sigma}^{-1} 
\right)
\right|
\leq 
C \epsilon
\label{equation_transient_model_error}
\end{equation}
where $I$ is the identity matrix.
\label{lemma_how_far_do_nonhomog_steady_states_lie_apart}
\end{lemma}
\begin{proof}
Define $\kappa$ as the left-hand side in equation \eqref{equation_transient_model_error}.
Let the vector $q \in \mathbb{R}^3$ be equal to $B_{\Sigma} b_{\Sigma}^{-1}$ except for its first element, which we choose equal to $\sigma^2/\lambda_{\Sigma}$.
Let the vector $v \in \mathbb{R}^3$ be chosen such that $q + v = B_{\Sigma}^{-1} b_{\Sigma}$. 
Then it holds that, using the linearity of the restriction operator $\mathcal R_{\mathrm{SDE}}$,
\begin{equation}
\begin{aligned}
\kappa
&\leq 
\sup_{t \in [0,T]}
\left| e^{\lambda_{\Sigma} t} \frac{\sigma^2}{\lambda_{\Sigma}}
-
\mathcal R_{\mathrm{SDE}} \left(
e^{B_{\Sigma}t} B_{\Sigma}^{-1} b_{\Sigma}^{-1} 
\right)
\right|
+
\left| \frac{\sigma^2}{\lambda_{\Sigma}}
-
\mathcal R_{\mathrm{SDE}} \left(
B_{\Sigma}^{-1} b_{\Sigma}^{-1}  \right) 
\right| \\ 
&\leq 
\sup_{t \in [0,T]}
\left| e^{\lambda_{\Sigma} t} \frac{\sigma^2}{\lambda_{\Sigma}}
-
\mathcal R_{\mathrm{SDE}} \left(
e^{B_{\Sigma}t} q  \right)
\right| 
+ 
\sup_{t \in [0,T]}
\left| 
\mathcal R_{\mathrm{SDE}}  \left(
e^{B_{\Sigma}t} v  \right)
\right|
+
\left| \frac{\sigma^2}{\lambda_{\Sigma}}
-
\mathcal R_{\mathrm{SDE}} \left(
B_{\Sigma}^{-1} b_{\Sigma}^{-1} \right) \right|,
\end{aligned}
\label{expression_kappa}
\end{equation}
where the supremum is not necessary in the last term since it is time-independent.

The first term in equation \eqref{expression_kappa} can be interpreted as the difference between the solution of a multiscale linear ODE and its reduced model, and can be bounded using lemma \ref{lemma_my_2013_eq_2_8} (equation \eqref{theorem_1}):
there exists a $\epsilon_0 \in (0,1)$ and constant $C$, independent of $\epsilon$, such that for all $\epsilon < \epsilon_0$
\begin{equation}
\begin{aligned}
&\sup_{t \in [0,T]}
\left| e^{\lambda_{\Sigma} t} \frac{\sigma^2}{\lambda_{\Sigma}}
-
\mathcal R_{\mathrm{SDE}} e^{B_{\Sigma}t} \left( q \right)
\right| \\
&\leq  
C \epsilon 
\left( 
\frac{\sigma^2}{\lambda_{\Sigma}} 
+ 
\norm{\mathcal R^{\perp}_{\mathrm{SDE}}  \left(
q \right) - A_{\Sigma}(\epsilon)^{-1} q_{\Sigma} 
\frac{\sigma^2}{\lambda_{\Sigma}}}
\right) \\
&\leq 
C \epsilon
\left(
\frac{\sigma^2}{\lambda_{\Sigma}} 
+
\norm{R^{\perp}_{\mathrm{SDE}} \left( q \right) }
+ 
\norm{A_{\Sigma}(\epsilon)^{-1} q_{\Sigma} 
\frac{\sigma^2}{\lambda_{\Sigma}}}
\right)
\end{aligned}
\end{equation}
the norm of $\mathcal R^{\perp}_{\mathrm{SDE}} (q) = \mathcal R^{\perp}_{\mathrm{SDE}} \left( B_{\Sigma}^{-1}b_{\Sigma} \right)$, 
can be bounded using lemma \ref{lemma_norm_inhomogeneity}, 
and for the last term we have 
$\norm{A_{\Sigma}^{-1}(\epsilon) q_{\Sigma} \frac{\sigma^2}{\lambda_{\Sigma}}} 
\leq
\norm{A_{\Sigma}^{-1}(\epsilon)} 
\norm{q_{\Sigma}}
\norm{\frac{\sigma^2}{\lambda_{\Sigma}}}$ where $\norm{A_{\Sigma}^{-1}(\epsilon)}$ is bounded using equation \eqref{equation_norm_A_Sigma_inverse}.
In conclusion, there exists a constant $C$, independent of $\epsilon$, such that 
\begin{equation}
\sup_{t \in [0,T]}
\left| e^{\lambda_{\Sigma} t} \frac{\sigma^2}{\lambda_{\Sigma}}
-
\mathcal R_{\mathrm{SDE}} \left( e^{B_{\Sigma}t} q \right)
\right| 
\leq 
C \epsilon
\end{equation}

\sloppy
The second term in equation \eqref{expression_kappa} can be bounded using lemma \ref{lemma_my_2013_eq_2_14} and lemma \ref{easy_corrolary}. 
More specifically, from lemma \ref{lemma_my_2013_eq_2_14} (equation \eqref{equation_properties_homogenised_system}) we know that 
$\sup_{t \in [0,T]} \left| \mathcal R_{\mathrm{SDE}} 
\left(
e^{B_{\Sigma}t} v \right) \right| 
\leq 
C \left( |\mathcal R_{\mathrm{SDE}} (v)| + \epsilon 
\norm{\mathcal R^{\perp}_{\mathrm{SDE}} (v)} \right)$. 
We can now use the fact that $|\mathcal R^{\perp} (v)|=0$ and from lemma \ref{easy_corrolary} we know that 
$\norm{ \mathcal R_{\mathrm{SDE}} (v) } 
= 
\left| \frac{\sigma^2}{\lambda_{\Sigma}}  - 
\mathcal R_{\mathrm{SDE}} 
\left( B_{\Sigma}^{-1} b_{\Sigma} \right)
 \right| \leq C \epsilon$. 
Thus, there exists a constant $C$, independent from $\epsilon$, such that 
\begin{equation}
\sup_{t \in [0,T]} \left| \mathcal R_{\mathrm{SDE}} \left( e^{B_{\Sigma}t} v \right) \right| \leq C \epsilon
\end{equation}

The last term in equation \eqref{expression_kappa} is bounded through lemma \ref{easy_corrolary}.
The combination of these bounds leads to equation \eqref{equation_transient_model_error}.
\end{proof}

\bibliographystyle{spmpsci}
\bibliography{refs.bib}

\end{document}